\documentclass[12pt, a4paper]{article}

\usepackage[utf8]{inputenc}
\usepackage[T1]{fontenc}
\usepackage[english]{babel}
\usepackage{geometry}
\usepackage{amsmath, amssymb, amsthm}
\usepackage{mathtools}
\usepackage{bm}
\usepackage{graphicx}
\usepackage{float}
\usepackage[table]{xcolor}
\usepackage{enumitem}
\usepackage{hyperref}
\hypersetup{colorlinks=true, linkcolor=blue!60!black, citecolor=blue!60!black, urlcolor=blue!60!black}
\usepackage[numbers]{natbib}
\usepackage{setspace}
\usepackage{parskip}
\usepackage{titlesec}
\usepackage{fancyhdr}
\usepackage{tabularx}
\usepackage{multirow}
\usepackage{algorithm}
\usepackage{algpseudocode}

\titleformat{\section}{\normalfont\large\bfseries}{\thesection}{1em}{}
\titlespacing*{\section}{0pt}{16pt}{6pt}
\titleformat{\subsection}{\normalfont\normalsize\bfseries}{\thesubsection}{1em}{}
\titlespacing*{\subsection}{0pt}{12pt}{4pt}

\newcommand{\R}{\mathbb{R}}
\newcommand{\norm}[1]{\left\|#1\right\|}

\theoremstyle{plain}
\newtheorem{proposition}{Proposition}
\theoremstyle{remark}
\newtheorem{remark}{Remark}
\theoremstyle{plain}

\newtheorem{theorem}{Theorem}
\theoremstyle{remark}

\begin{document}
\thispagestyle{empty}

\begin{center}
\vspace*{1cm}
\includegraphics[height=2.75cm]{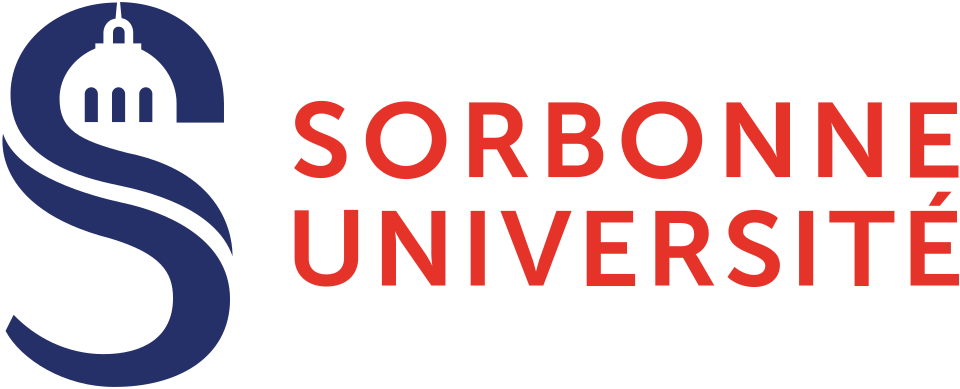}
\hspace{2.8cm}
\includegraphics[height=2.8cm]{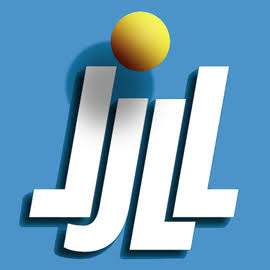}
\vspace{2.5cm}

{\large Summer Research Internship Thesis, Laboratoire Jacques-Louis Lions}
\vspace{-0.5mm}
\rule{\linewidth}{0.7pt}

{\Large \bfseries Real-Time Model Predictive Control Algorithms \\[7pt]
for Autonomous Spacecraft Guidance}

\rule{\linewidth}{0.7pt}

\begin{tabular}{p{12.5cm} p{2.5cm}}
{\large \textit{Author:}}   & {\large \textit{Supervisor: \ }} \\[7pt]
{\large Mohammed-Adnane GARAB} \quad & {\large Max Cerf \ }             \\
\end{tabular}
\vfill

\vspace{1mm}

\end{center}

\vfill
\begin{center}
{\small Undergraduate Research Intern} \\[2pt]
{\normalsize Summer 2026}
\end{center}
\newpage
{
\setstretch{1}
\setlength{\parskip}{0pt}
\tableofcontents
}
\newpage

\newpage

 \section*{Acknowledgements}
\vspace{3mm}

\textit{I would like to express my sincere gratitude to Max Cerf for
his supervision throughout this internship, for the rigour and clarity
of his guidance, and for the numerous discussions that greatly
contributed to the direction and depth of this work.}

\newpage

\section{Introduction}

Spacecraft rendezvous and proximity operations are among the most technically
demanding challenges in modern space mission design. The goal is to steer an
active vehicle, the \emph{chaser}, towards a passive or semi-passive \emph{target}
in orbit around the Earth, and to dock with it at a prescribed relative
configuration and velocity. Concrete examples include the resupply of the
International Space Station, active debris removal campaigns, and the Mars
Sample Return capture scenario currently studied by ESA and NASA.

Since round-trip communication delays between the spacecraft and the ground can
reach several minutes, fully autonomous onboard guidance is a strict operational
requirement. The chaser must plan, replan, and execute its trajectory without any
human supervision, while simultaneously satisfying a rich set of physical and
safety constraints: bounded thrust, a line-of-sight cone imposed by the navigation
sensors, a soft-docking velocity profile near the docking port, and collision
avoidance margins in the event of a thruster failure.

Traditional guidance architectures rely on pre-computed open-loop manoeuvre
sequences with ad hoc error corrections. These approaches are straightforward to
certify, but they are fundamentally limited in robustness: a single perturbation
event, such as an unexpected atmospheric drag force or a thrust calibration error,
can cause the chaser to drift from its planned path and compromise the mission.
This has motivated, over the past two decades, a growing body of research on
closed-loop, optimisation-based guidance laws that handle constraints explicitly
and adapt online to perturbations. Among these, Model Predictive Control (MPC)
has established itself as the reference framework, owing to its ability to combine
systematic constraint handling, receding-horizon online replanning, and robust
feedback in a single coherent formulation.

Model predictive control has already been investigated extensively for
spacecraft rendezvous, both in nominal and robust settings. Early works
demonstrated safe and constraint-aware rendezvous trajectories
\cite{Breger2008}, while later studies developed MPC guidance laws for
rendezvous and proximity operations under disturbances \cite{DiCairano2012}.
Design and implementation studies further established practical MPC
architectures for spacecraft rendezvous, including constrained thrust
allocation and terminal capture strategies
\cite{Hartley2012,Hartley2013}. Tutorial and implementation-oriented
treatments further clarified the practical design choices involved in
spacecraft rendezvous MPC \cite{Hartley2015}. More recent contributions
extended the framework to tube-based robust rendezvous with uncertainty
adaptation \cite{Specht2023,Oestreich2023} and to nonlinear guidance via
successive convexification on more challenging orbital regimes
\cite{Szmuk2025}.

The purpose of this work is to study and compare four families of MPC algorithms
for autonomous spacecraft guidance:
\begin{enumerate}[label=\roman*)]
  \item Linear MPC,
  \item Tube MPC,
  \item Fast and Embedded MPC,
  \item Successive Convexification.
\end{enumerate}
We first introduce the Clohessy-Wiltshire-Hill relative-motion model, then
describe the four MPC families and their numerical solvers, and report
numerical results on a planar orbital rendezvous benchmark. Tube MPC stands
out for its formal robustness guarantee at no extra online cost, and the
remainder of this work focuses on it specifically. Its validity is assessed
across five standard rendezvous manoeuvres under a single fixed controller
configuration, before its condensed matrices are simplified through a
polynomial approximation and the fuel-accuracy trade-off induced by the cost
weights is examined, yielding a provably correct, closed-form bound on the
worst-case tracking error. The framework is then extended beyond a fixed
circular reference orbit to an arbitrary trajectory tracked through an
online-recomputed linearisation. Finally, the feasibility of the resulting
pipeline for onboard execution is assessed by re-implementing its core
numerical routines from first principles, without reliance on external
scientific libraries.

\section{The Clohessy-Wiltshire-Hill Model}

\subsection{Relative dynamics}

The dynamics of the chaser relative to the target, valid when the separation is
small compared to the orbital radius, are captured by the linearised
Clohessy-Wiltshire-Hill (CWH) equations \cite{Clohessy1960}. Letting
$\bm{x} = (\delta x,\,\delta y,\,\delta\dot{x},\,\delta\dot{y})^\top \in \R^4$
denote the relative state in the Local Vertical Local Horizontal (LVLH) frame
and $\bm{u} = (u_x, u_y)^\top \in \R^2$ the control acceleration, the
continuous-time model reads
\begin{equation}
  \dot{\bm{x}}(t) = A_c\,\bm{x}(t) + B_c\,\bm{u}(t),
\end{equation}
with
\begin{equation}
  A_c = \begin{pmatrix}
    0    & 0 & 1    & 0  \\
    0    & 0 & 0    & 1  \\
    3n^2 & 0 & 0    & 2n \\
    0    & 0 & -2n  & 0
  \end{pmatrix},
  \qquad
  B_c = \begin{pmatrix}
    0 & 0 \\ 0 & 0 \\ 1 & 0 \\ 0 & 1
  \end{pmatrix},
\end{equation}
where $n = \sqrt{\mu/R_0^3}$ is the mean orbital rate, $\mu = 3.986\times10^{14}$\,m$^3$/s$^2$
is Earth's gravitational parameter, and $R_0$ is the orbital radius of the
target (at 500\,km altitude, $n = 1.107\times10^{-3}$\,rad/s).

Discretised with sampling period $T_s$ via a \textbf{zero-order hold}
(meaning the control $\bm{u}$ is held constant over each interval
$[t_k, t_{k+1})$), the model becomes
\begin{equation}
  \label{eq:CW_disc}
  \bm{x}_{k+1} = A_d\,\bm{x}_k + B_d\,\bm{u}_k,
  \qquad A_d = e^{A_c T_s},\quad
  B_d = \int_0^{T_s} e^{A_c\tau} B_c\,\mathrm{d}\tau.
\end{equation}

\textbf{Derivation of $A_d$ and $B_d$.}
The discrete matrices arise from solving the continuous ODE exactly over one
sampling interval. Fix a step $k$ and consider
$\dot{\bm{x}}(t) = A_c\,\bm{x}(t) + B_c\,\bm{u}(t)$ on $[t_k, t_{k+1}]$.
Multiply both sides by the integrating factor $e^{-A_c t}$:
\begin{equation}
  \frac{\mathrm{d}}{\mathrm{d}t}\!\left[e^{-A_c t}\bm{x}(t)\right]
  = e^{-A_c t}B_c\,\bm{u}(t).
\end{equation}
Integrating from $t_k$ to $t_{k+1} = t_k + T_s$ and multiplying by
$e^{A_c t_{k+1}}$:
\begin{equation}
  \bm{x}(t_{k+1})
  = e^{A_c T_s}\,\bm{x}(t_k)
  + e^{A_c t_{k+1}}\!\int_{t_k}^{t_{k+1}} e^{-A_c\tau}\,B_c\,\bm{u}(\tau)\,\mathrm{d}\tau.
\end{equation}
Since $\bm{u}(\tau) = \bm{u}_k$ is constant (zero-order hold), setting
$s = t_{k+1} - \tau$ gives
$e^{A_c t_{k+1}}\int_{t_k}^{t_{k+1}}e^{-A_c\tau}B_c\,\mathrm{d}\tau
= \int_0^{T_s}e^{A_c s}B_c\,\mathrm{d}s$,
which yields exactly $\bm{x}_{k+1} = A_d\bm{x}_k + B_d\bm{u}_k$.
These are \emph{exact}, not approximations. Both are computed once offline
using \texttt{scipy.linalg.expm}.

In the presence of bounded disturbances (atmospheric drag, actuation errors),
the real dynamics are
\begin{equation}
  \bm{x}_{k+1} = A_d\,\bm{x}_k + B_d\,\bm{u}_k + \bm{\xi}_k,
  \qquad \bm{\xi}_k \in \mathcal{W},
\end{equation}
where $\mathcal{W} \subset \R^4$ is a compact convex set bounding the
worst-case disturbance at each step.

\subsection{Tracking objective}

Given a nominal trajectory $\{\bm{x}_k^{\mathrm{ref}}\}$ pre-computed on the
ground, the onboard guidance law must correct the control command in real time
so that the chaser remains within a corridor
$\norm{\bm{x}_k - \bm{x}_k^{\mathrm{ref}}} \leq \varepsilon$ while minimising
a cost of the form $\norm{\bm{u}} + \norm{\bm{e}}$, where
$\bm{e} = \bm{x} - \bm{x}^{\mathrm{ref}}$ is the tracking error. The parameter
$\varepsilon > 0$ is the \emph{corridor width}.

\section{MPC Algorithms for Spacecraft Guidance}

\subsection{General formulation}

The general MPC optimisation solved at each time step $k$ is
\begin{equation}
  \label{eq:MPC}
  \min_{\{\bm{u}_i\}_{i=k}^{k+N-1}}
  \; V_f(\bm{x}_{k+N}) + \sum_{i=k}^{k+N-1} \ell(\bm{x}_i, \bm{u}_i)
\end{equation}
subject to the dynamics \eqref{eq:CW_disc}, input constraints
$\bm{u}_i \in \mathcal{U}$, state constraints $\bm{x}_i \in \mathcal{X}$,
and a terminal constraint $\bm{x}_{k+N} \in \mathcal{X}_f$. The notation is
as follows.
\begin{itemize}
  \item $\ell(\bm{x}_i, \bm{u}_i)$ is the \textbf{stage cost}: it penalises
        the tracking error and the control effort at each step $i$. A typical
        choice is
        $\ell = (\bm{x}_i - \bm{x}_i^{\mathrm{ref}})^\top Q\,
        (\bm{x}_i - \bm{x}_i^{\mathrm{ref}}) + \bm{u}_i^\top R\,\bm{u}_i$.
  \item $V_f(\bm{x}_{k+N})$ is the \textbf{terminal cost}: it penalises the
        final predicted state $\bm{x}_{k+N}$ at the end of the horizon, chosen
        to compensate for the finite horizon and ensure stability.
  \item $N \geq 1$ is the \textbf{prediction horizon}: the number of future
        steps considered.
  \item $\mathcal{U} = \{\bm{u} : \norm{\bm{u}}_\infty \leq u_{\max}\}$ is
        the input constraint set (thrust bound).
  \item $\mathcal{X}$ is the corridor constraint set
        ($\norm{\bm{x} - \bm{x}^{\mathrm{ref}}} \leq \varepsilon$).
  \item $\mathcal{X}_f$ is a small neighbourhood of the docking target,
        chosen to ensure the chaser is close enough to the goal at the end
        of the horizon.
\end{itemize}
Only the first element $\bm{u}_k^\star$ of the optimal sequence is applied;
the optimisation is repeated at $k+1$ with updated state measurements. This
\emph{receding-horizon} strategy turns an open-loop plan into a closed-loop
feedback law. Under standard terminal conditions on $V_f$ and $\mathcal{X}_f$,
this scheme guarantees recursive feasibility and asymptotic stability of the
nominal trajectory \cite{Mayne2000}.

\begin{table}[H]
\centering
\caption{Summary of MPC families for autonomous spacecraft guidance.}
\vspace{4mm}
\label{tab:comparison}
\renewcommand{\arraystretch}{1.4}
\small
\begin{tabular}{lllll}
\hline
\textbf{Method} & \textbf{Core problem} & \textbf{Robustness}
  & \textbf{Comp.\ cost} & \textbf{Corridor guarantee} \\
\hline
Linear MPC & QP             & Implicit (feedback)
  & Low      & Nominal only \\
Tube MPC   & QP (tightened) & Formal ($\forall\bm{\xi}\in\mathcal{W}$)
  & Low      & By design \\
Fast MPC   & QP (optimised) & Same as base
  & Very low & Inherited \\
SCvx       & SOCP (iter.)   & Local
  & High     & Nonlinear \\
\hline
\end{tabular}
\end{table}

\subsection{Linear MPC}
\label{sec:linear_mpc}

Take the quadratic stage cost
\begin{equation}
  \ell(\bm{x}_i, \bm{u}_i)
  = (\bm{x}_i - \bm{x}_i^{\mathrm{ref}})^\top Q\,(\bm{x}_i - \bm{x}_i^{\mathrm{ref}})
  + \bm{u}_i^\top R\,\bm{u}_i,
\end{equation}
where $Q \in \R^{4\times4}$, $Q \succeq 0$ weights the tracking error and
$R \in \R^{2\times2}$, $R \succ 0$ weights the control effort. Stack all $N$
predicted states and controls into
\begin{equation}
  \bm{X} =
  \begin{pmatrix}\bm{x}_{k+1}\\\vdots\\\bm{x}_{k+N}\end{pmatrix}
  \in \R^{4N},
  \qquad
  \bm{v} =
  \begin{pmatrix}\bm{u}_k\\\vdots\\\bm{u}_{k+N-1}\end{pmatrix}
  \in \R^{2N}.
\end{equation}
Because the dynamics are linear, applying $\bm{x}_{i+1} = A_d\bm{x}_i +
B_d\bm{u}_i$ repeatedly from $\hat{\bm{x}}_k$ gives:
\begin{align*}
  \bm{x}_{k+1} &= A_d\hat{\bm{x}}_k + B_d\bm{u}_k, \\
  \bm{x}_{k+2} &= A_d^2\hat{\bm{x}}_k + A_dB_d\bm{u}_k + B_d\bm{u}_{k+1}, \\
  \bm{x}_{k+3} &= A_d^3\hat{\bm{x}}_k + A_d^2B_d\bm{u}_k
                  + A_dB_d\bm{u}_{k+1} + B_d\bm{u}_{k+2}, \\
               &\;\;\vdots
\end{align*}
The pattern is clear: each step adds one more power of $A_d$. Collecting all
$N$ steps into the single matrix equation
$\bm{X} = \mathcal{A}\hat{\bm{x}}_k + \mathcal{B}\bm{v}$:
\begin{equation}
  \label{eq:condensed}
  \mathcal{A} =
  \begin{pmatrix}
    A_d \\ A_d^2 \\ A_d^3 \\ \vdots \\ A_d^N
  \end{pmatrix}
  \in \R^{4N\times4},
  \quad
  \mathcal{B} =
  \begin{pmatrix}
    B_d          & 0            & 0      & \cdots & 0   \\
    A_dB_d       & B_d          & 0      & \cdots & 0   \\
    A_d^2B_d     & A_dB_d       & B_d    & \cdots & 0   \\
    \vdots       & \vdots       & \ddots & \ddots & \vdots \\
    A_d^{N-1}B_d & A_d^{N-2}B_d & \cdots & A_dB_d & B_d
  \end{pmatrix}
  \in \R^{4N\times2N}.
\end{equation}
The $(i,j)$-th block of $\mathcal{B}$ is $A_d^{i-j}B_d$ for $i \geq j$ and
$0$ for $i < j$. Setting
\begin{equation}
  \bm{c} = \mathcal{A}\hat{\bm{x}}_k - \bm{X}^{\mathrm{ref}} \in \R^{4N}
\end{equation}
(a constant vector for a given current state $\hat{\bm{x}}_k$ and reference
$\bm{X}^{\mathrm{ref}}$), the condensed relation reads
$\bm{X} - \bm{X}^{\mathrm{ref}} = \mathcal{B}\bm{v} + \bm{c}$.
Substituting into the stage-cost sum:
\begin{align}
  J
  &= (\bm{X} - \bm{X}^{\mathrm{ref}})^\top \bar{Q}\,
     (\bm{X} - \bm{X}^{\mathrm{ref}})
   + \bm{v}^\top\bar{R}\bm{v} \notag\\
  &= (\mathcal{B}\bm{v} + \bm{c})^\top \bar{Q}\,
     (\mathcal{B}\bm{v} + \bm{c})
   + \bm{v}^\top\bar{R}\bm{v}. \notag
\end{align}
Expanding the quadratic form
$(\mathcal{B}\bm{v}+\bm{c})^\top\bar{Q}(\mathcal{B}\bm{v}+\bm{c})
= \bm{v}^\top\mathcal{B}^\top\bar{Q}\mathcal{B}\bm{v}
+ 2\bm{c}^\top\bar{Q}\mathcal{B}\bm{v}
+ \bm{c}^\top\bar{Q}\bm{c}$
and grouping terms in $\bm{v}$:
\begin{align}
  J
  &= \bm{v}^\top
     \underbrace{(\mathcal{B}^\top\bar{Q}\mathcal{B}+\bar{R})}_{H}
     \bm{v}
   + 2\,\underbrace{(\mathcal{B}^\top\bar{Q}\bm{c})^\top}_{\bm{f}^\top}
     \bm{v}
   + \underbrace{\bm{c}^\top\bar{Q}\bm{c}}_{\text{constant in }\bm{v}}.
\end{align}
where $\bar{Q} = \mathrm{blkdiag}(Q,\ldots,Q,P) \in \R^{4N\times4N}$ and
$\bar{R} = \mathrm{blkdiag}(R,\ldots,R) \in \R^{2N\times2N}$. Here $P \succeq 0$
is the \textbf{terminal weight matrix}, typically chosen as the solution of the
discrete algebraic Riccati equation, which makes the terminal cost
$V_f(\bm{x}) = \bm{x}^\top P\bm{x}$ equivalent to the infinite-horizon LQR
cost and ensures closed-loop stability.

Dropping the constant term, problem \eqref{eq:MPC} reduces to the
\textbf{Quadratic Program (QP)}:
\begin{equation}
  \label{eq:qp}
  \min_{\bm{v}}\;\tfrac{1}{2}\bm{v}^\top H\bm{v} + \bm{f}^\top\bm{v}
  \quad\text{subject to}\quad G\bm{v} \leq \bm{g},
\end{equation}
with
\begin{equation}
  \label{eq:Hf}
  H = \mathcal{B}^\top\bar{Q}\mathcal{B} + \bar{R} \in \R^{2N\times2N},
  \qquad
  \bm{f} = \mathcal{B}^\top\bar{Q}
            \!\left(\mathcal{A}\hat{\bm{x}}_k - \bm{X}^{\mathrm{ref}}\right)
            \in \R^{2N}.
\end{equation}
The matrix $G \in \R^{p\times2N}$ and the vector $\bm{g} \in \R^p$ encode all
linear inequality constraints: the input box $\norm{\bm{u}_i}_\infty \leq u_{\max}$
and, rewritten via \eqref{eq:condensed}, the corridor constraint on $\bm{X}$.
Since $R \succ 0$ implies $H \succ 0$, this QP is strictly convex and has a
unique global minimiser. The matrices $\mathcal{A}$, $\mathcal{B}$, $H$, and
$G$ depend only on $A_d$, $B_d$, $Q$, $R$, $P$, and $N$, all fixed before
the mission. They are computed once offline. At each guidance step, only $\bm{f}$
needs to be updated, since it depends linearly on $\hat{\bm{x}}_k$.

\begin{remark}
\label{rem:kappa_growth}
Because all eigenvalues of $A_d$ have magnitude 1, the powers $A_d^j$ never
shrink, so all columns of $\mathcal{B}$ stay large across all $N$ steps.
This causes the condition number
$\kappa(H) = \lambda_{\max}(H)/\lambda_{\min}(H)$ to blow up: with $N=40$,
$\lambda_{\max}(Q) \approx 3\times10^5$, and $\lambda_{\min}(R)=100$,
one obtains $\kappa(H) \approx 1.4\times10^7$. This is a direct consequence
of the marginal stability of CWH, and it explains why first-order solvers
stall on this problem.
\end{remark}

Hartley \cite{Hartley2015} compared four cost structures for long-range
rendezvous: the quadratic cost yields smooth trajectories, while the 1-norm
cost promotes sparse bang-off-bang thrust profiles with better fuel efficiency.
Di~Cairano, Park and Kolmanovsky \cite{DiCairano2012} validated Linear MPC in
closed-loop simulation with disturbances up to 10\% of maximum thrust, achieving
docking errors below 1.2\,cm. From a numerical optimisation viewpoint, the
resulting online problem belongs to the standard class of strictly convex
quadratic programmes long studied in MPC, for which both active-set and
interior-point methods provide reliable baseline solvers \cite{Rao1998,Rawlings2017}.
The key limitation is the absence of a formal corridor guarantee: the
receding-horizon feedback corrects disturbances reactively, but nothing prevents
$\norm{\bm{e}_k} \leq \varepsilon$ from being temporarily violated.

\subsection{Tube MPC}

Tube MPC \cite{Mayne2005, Langson2004} fixes the corridor guarantee problem at no extra
online cost. The true state is split into two parts,
\begin{equation}
  \bm{x}_k = \bm{z}_k + \bm{e}_k,
\end{equation}
where $\bm{z}_k \in \R^4$ is a nominal state driven by a planned input
$\bm{v}_k$, and $\bm{e}_k = \bm{x}_k - \bm{z}_k$ is the error caused by
disturbances. The actual thrust applied to the spacecraft is
\begin{equation}
  \label{eq:tube_control}
  \bm{u}_k = \bm{v}_k - K\bm{e}_k,
\end{equation}
where $K \in \R^{2\times4}$ is a pre-computed feedback gain and
$K\bm{e}_k$ is a real-time correction that pushes the chaser back
towards the nominal trajectory whenever it drifts.

The gain $K \in \R^{2\times4}$ is found by solving the following
minimisation problem, called the \textbf{Linear Quadratic Regulator
(LQR)}: find the linear feedback $\bm{u}_k = -K\bm{e}_k$ that
minimises the total cost over an infinite horizon,
\begin{equation}
  J_\infty = \sum_{k=0}^{\infty}
  \bigl(\bm{e}_k^\top Q\,\bm{e}_k + \bm{u}_k^\top R\,\bm{u}_k\bigr),
\end{equation}
subject to the error dynamics $\bm{e}_{k+1} = A_d\bm{e}_k +
B_d\bm{u}_k$, where $A_d \in \R^{4\times4}$ describes how the state
evolves in one step without thrust, $B_d \in \R^{4\times2}$ describes
the effect of a thrust pulse on the state, $Q \in \R^{4\times4}$
penalises position and velocity errors, and $R \in \R^{2\times2}$
penalises thrust usage.

To minimise $J_\infty$, we look for the optimal $\bm{u}_k$ at each
step $k$. At step $k$, the cost to pay from now until infinity can be
written as two parts: the cost of this step, plus the cost of all
future steps:
\begin{equation}
  J_\infty
  = \underbrace{\bm{e}_k^\top Q\bm{e}_k + \bm{u}_k^\top R\bm{u}_k}
    _{\text{cost at step }k}
  + \underbrace{\sum_{j=k+1}^{\infty}
    \bigl(\bm{e}_j^\top Q\bm{e}_j + \bm{u}_j^\top R\bm{u}_j\bigr)}
    _{\text{cost from step }k+1\text{ onwards}}.
\end{equation}
The key observation is that the second part depends only on
$\bm{e}_{k+1}$, not on $\bm{e}_k$ directly. We therefore minimise
over $\bm{u}_k$ by differentiating the cost of step $k$ plus the
cost from $k+1$ onwards with respect to $\bm{u}_k$, and setting the
result to zero. Substituting $\bm{e}_{k+1} = A_d\bm{e}_k +
B_d\bm{u}_k$ and carrying out the differentiation gives:
\begin{equation}
  \frac{\partial}{\partial \bm{u}_k}
  \Bigl[
    \bm{u}_k^\top R\bm{u}_k
    + (A_d\bm{e}_k + B_d\bm{u}_k)^\top P
      (A_d\bm{e}_k + B_d\bm{u}_k)
  \Bigr] = 0,
\end{equation}
where $P \in \R^{4\times4}$ summarises the optimal future cost:
$\bm{e}_{k+1}^\top P\bm{e}_{k+1}$ is the best achievable cost
from step $k+1$ onwards when starting from $\bm{e}_{k+1}$.
Computing the derivative and solving for $\bm{u}_k^\star$:
\begin{align}
  2R\bm{u}_k + 2B_d^\top P(A_d\bm{e}_k + B_d\bm{u}_k) &= 0 \notag\\
  (R + B_d^\top P B_d)\bm{u}_k &= -B_d^\top P A_d\bm{e}_k \notag\\
  \bm{u}_k^\star &= -\underbrace{(R + B_d^\top P B_d)^{-1}
                    B_d^\top P A_d}_{K}\,\bm{e}_k.
\end{align}
This gives the optimal gain directly:
\begin{equation}
  \label{eq:K}
  K = (R + B_d^\top P B_d)^{-1} B_d^\top P A_d \in \R^{2\times4}.
\end{equation}
The matrix $P$ is determined by a self-consistency argument.
We assumed that the optimal cost from state $\bm{e}_k$ has the
quadratic form $\bm{e}_k^\top P\bm{e}_k$. We split the total
cost into the cost at step $k$ and the cost from step $k+1$
onwards:
\begin{equation}
  J_\infty
  = \underbrace{\bm{e}_k^\top Q\bm{e}_k
    + \bm{u}_k^\top R\bm{u}_k}_{\text{cost at step }k}
  + \underbrace{\sum_{j=k+1}^{\infty}
    \bigl(\bm{e}_j^\top Q\bm{e}_j + \bm{u}_j^\top R\bm{u}_j\bigr)}_
    {=\;\bm{e}_{k+1}^\top P\bm{e}_{k+1}
     \text{ by assumption}}.
\end{equation}

Substituting the optimal feedback $\bm{u}_k^\star = -K\bm{e}_k$
into the first term:
\begin{equation}
  \bm{u}_k^\top R\bm{u}_k
  = (-K\bm{e}_k)^\top R(-K\bm{e}_k)
  = \bm{e}_k^\top K^\top R K\bm{e}_k.
\end{equation}
Substituting $\bm{e}_{k+1} = A_{cl}\bm{e}_k$ into the second term:
\begin{equation}
  \bm{e}_{k+1}^\top P\bm{e}_{k+1}
  = (A_{cl}\bm{e}_k)^\top P(A_{cl}\bm{e}_k)
  = \bm{e}_k^\top A_{cl}^\top P A_{cl}\bm{e}_k.
\end{equation}
The total cost from step $k$ is therefore:
\begin{equation}
  J_\infty
  = \bm{e}_k^\top
    \bigl(Q + K^\top R K + A_{cl}^\top P A_{cl}\bigr)
    \bm{e}_k.
\end{equation}
For this to equal $\bm{e}_k^\top P\bm{e}_k$ for all $\bm{e}_k$,
the matrix inside the parentheses must equal $P$:
\begin{equation}
  P = Q + K^\top R K + A_{cl}^\top P A_{cl}.
\end{equation}
Substituting $K = (R + B_d^\top P B_d)^{-1}B_d^\top P A_d$ and
$A_{cl} = A_d - B_dK$ and simplifying yields the equation that
$P$ must satisfy, known as the Discrete Algebraic Riccati Equation:
\begin{equation}
  \label{eq:DARE}
  P = Q + A_d^\top P A_d
    - A_d^\top P B_d\,(R + B_d^\top P B_d)^{-1}B_d^\top P A_d,
\end{equation}
solved once offline using \texttt{scipy.linalg.solve\_discrete\_are}.

\begin{proposition}[Decoupling of nominal and error dynamics]
Under the control law \eqref{eq:tube_control}, the nominal state $\bm{z}_k$ and
the error $\bm{e}_k$ evolve independently:
\begin{align}
  \bm{z}_{k+1} &= A_d\bm{z}_k + B_d\bm{v}_k, \label{eq:nom_dyn}\\
  \bm{e}_{k+1} &= (A_d - B_dK)\bm{e}_k + \bm{\xi}_k. \label{eq:err_dyn}
\end{align}
\end{proposition}

\begin{proof}
Substitute $\bm{x}_k = \bm{z}_k + \bm{e}_k$ and
$\bm{u}_k = \bm{v}_k - K\bm{e}_k$ into
$\bm{x}_{k+1} = A_d\bm{x}_k + B_d\bm{u}_k + \bm{\xi}_k$:
\begin{align*}
  \bm{x}_{k+1}
  &= A_d(\bm{z}_k + \bm{e}_k) + B_d(\bm{v}_k - K\bm{e}_k) + \bm{\xi}_k \\
  &= \underbrace{(A_d\bm{z}_k + B_d\bm{v}_k)}_{\bm{z}_{k+1}}
   + \underbrace{(A_d - B_dK)\bm{e}_k + \bm{\xi}_k}_{\bm{e}_{k+1}}.
\end{align*}
Reading off the two terms gives \eqref{eq:nom_dyn} and \eqref{eq:err_dyn}.
\end{proof}

Since $A_{cl} = A_d - B_dK$ is stable, the error $\bm{e}_k$
remains bounded regardless of the disturbances. More precisely,
there exists a bounded set $\mathcal{Z} \subset \R^4$, called the
\textbf{Robust Positively Invariant (RPI) set}
\cite{Rakovic2005RPI}, such that if
$\bm{e}_0 \in \mathcal{Z}$, then $\bm{e}_k \in \mathcal{Z}$ for
all $k \geq 0$, whatever disturbances $\bm{\xi}_k \in \mathcal{W}$
occur. Intuitively, $\mathcal{Z}$ is the worst-case tube around
the nominal trajectory that the error will never leave. Its radius
$\bar{z}$ is computed once offline.

Since $\bm{x}_k = \bm{z}_k + \bm{e}_k$ and $\bm{e}_k$ is at most
$\bar{z}$ in every direction, keeping the true state inside the
corridor $\|\bm{x}_k - \bm{x}_k^{\rm ref}\| \leq \varepsilon$
for all possible disturbances reduces to keeping the nominal state
inside a \textbf{tightened corridor} of half-width
$\varepsilon - \bar{z}$:
\begin{equation}
  \label{eq:tightened}
  \|\bm{z}_k - \bm{x}_k^{\rm ref}\| \leq \varepsilon - \bar{z}.
\end{equation}
Similarly, the nominal control is restricted to a tightened input
set $\tilde{\mathcal{U}}$ of radius $u_{\rm max} - \bar{u}$, where
$\bar{u} = \|K\|_\infty \bar{z}$ accounts for the worst-case
correction $K\bm{e}_k$.

The \textbf{Tube MPC algorithm} then proceeds as follows at each
step $k$:
\begin{enumerate}[label=\roman*)]
  \item Measure the true state $\bm{x}_k$ and compute the current
        error $\bm{e}_k = \bm{x}_k - \bm{z}_k$.
  \item Solve the nominal MPC optimisation \eqref{eq:qp} with the
        tightened corridor \eqref{eq:tightened} and tightened input
        set $\tilde{\mathcal{U}}$, to obtain the nominal control
        sequence $\{v_k, \ldots, v_{k+N-1}\}$.
  \item Apply the total control
        $\bm{u}_k = \bm{v}_k - K\bm{e}_k$ to the spacecraft.
  \item Advance the nominal state:
        $\bm{z}_{k+1} = A_d\bm{z}_k + B_d\bm{v}_k$.
\end{enumerate}
The tightening \eqref{eq:tightened} is computed once offline.
The online cost per step is identical to Linear MPC: same matrix
$H$, same ADMM solver. The formal corridor guarantee
$\|\bm{x}_k - \bm{x}_k^{\rm ref}\| \leq \varepsilon$ holds for
all realisations of the disturbance, at no additional computational
cost. Specht, Bishnoi and Lampariello \cite{Specht2023} applied
Tube MPC to rendezvous with freely tumbling targets. Oestreich,
Linares and Gondhalekar \cite{Oestreich2023} proposed an adaptive
variant that identifies $\mathcal{W}$ online from observed residuals,
and Tranos, Russo and Proutiere \cite{Tranos2022} proposed a related
self-tuning approach that adapts the tube itself online via
least-squares estimation.

While this work restricts to the linear CWH dynamics, the tube
framework extends naturally to nonlinear systems
\cite{MayneKerrigan2007,MayneKerrigan2011}, where the fixed gain $K$
is replaced by an ancillary nonlinear MPC controller enforcing the
trajectory to remain within a tube around a nominal reference. Recent
extensions replace the linear model by a globally linear Koopman
representation to handle nonlinear dynamics within the same tube
formulation \cite{Zhang2022Koopman}, and extend the framework to
systems with input delay \cite{Zhou2024Delay}, directly relevant to
the communication delays discussed in the introduction. Recent
extensions also address more structured uncertainty descriptions.
In particular, tube-based guaranteed-cost MPC has been proposed for linear
systems with parametric uncertainties \cite{Massera2020}, while scalable tube
MPC constructions based on ellipsoidal sets provide a computationally efficient
alternative for high-dimensional uncertain linear systems \cite{Parsi2022}.
Tutorial expositions of tube-based robust MPC can be found in
\cite{Herceg2015,Rakovic2008}, while simple output-feedback tube-MPC
variants have also been proposed \cite{Lorenzetti2019}.

\subsection{Fast and Embedded MPC}

Spacecraft guidance computers are slow and cannot run expensive matrix
factorisations at each guidance step. Fast MPC avoids this by using first-order
methods: algorithms that only require matrix-vector products of the form
$H\bm{v}$, with no factorisation. This philosophy is central to real-time and
embedded MPC, where the objective is not only to solve the QP accurately, but
to do so within a certified and very small computation budget
\cite{Richter2009,Richter2012}.

The gradient of the objective
$J(\bm{v}) = \tfrac{1}{2}\bm{v}^\top H\bm{v} + \bm{f}^\top\bm{v}$ is
$\nabla J(\bm{v}) = H\bm{v} + \bm{f}$, with Lipschitz
constant $L = \lambda_{\max}(H)$. The projected gradient step is
\begin{equation}
  \label{eq:grad_step}
  \bm{v}^{(j+1)}
  = \Pi_{\mathcal{U}}\!\left(
      \bm{v}^{(j)} - \frac{1}{L}(H\bm{v}^{(j)} + \bm{f})
    \right),
\end{equation}
where $\Pi_{\mathcal{U}}$ is the projection onto $\mathcal{U}$: for a box
constraint, this is simply an elementwise clip to $[-u_{\max}, u_{\max}]$.

Nesterov acceleration replaces the step from $\bm{v}^{(j)}$ by a step from an
extrapolated point,
\begin{equation}
  \bm{y}^{(j)} = \bm{v}^{(j)}
  + \frac{j-1}{j+2}\bigl(\bm{v}^{(j)} - \bm{v}^{(j-1)}\bigr),
\end{equation}
lifting the convergence rate from $O(1/j)$ to $O(1/j^2)$. For $H \succ 0$,
the suboptimality gap $J(\bm{v}^{(j)}) - J^\star$ is multiplied at each
iteration by the factor
\begin{equation}
  \rho = \frac{\sqrt{\kappa(H)}-1}{\sqrt{\kappa(H)}+1} < 1.
\end{equation}
The number of iterations to reach suboptimality $\delta$ satisfies
\cite{Richter2012}:
\begin{equation}
  j \leq
  \left\lceil
    \sqrt{\kappa(H)}\;\ln\!\frac{2\,J(\bm{v}^{(0)})}{\delta}
  \right\rceil.
\end{equation}
This bound is computable before the mission, guaranteeing that the solver
finishes within a fixed number of matrix-vector products regardless of the
specific problem instance.

Bashnick and Ulrich \cite{Bashnick2023} demonstrated sub-10\,ms solve times
for a rendezvous QP on an embedded processor. Breger and How
\cite{Breger2008} combined a fast LP solver with passive safety constraints,
guaranteeing collision avoidance even in the event of a complete thruster
failure. Beyond the basic accelerated projected-gradient method, several fast
variants have been proposed for real-time MPC, including dual fast-gradient
schemes for distributed and embedded implementations
\cite{Giselsson2013a,Giselsson2013b,Ferranti2015} and proximal-gradient
formulations tailored to constrained predictive control \cite{Patrinos2016}.

Related fast first-order schemes have also been developed for MPC problems with
both input-rate and amplitude constraints \cite{Kempf2020}. At the hardware
level, tailored embedded optimisation pipelines can reach extremely high update
rates, up to the megahertz range in dedicated implementations
\cite{Jerez2014,Frison2016}. For explicit and parametric MPC implementations,
software tools such as the Multi-Parametric Toolbox 3.0 provide a standard
benchmark environment for controller synthesis and deployment
\cite{Herceg2013}.

\subsection{Successive Convexification}

For highly elliptic orbits, large out-of-plane motion, or problems with
non-convex thrust constraints such as $\norm{\bm{u}} \geq u_{\min}$, the true
dynamics are nonlinear: $\bm{x}_{k+1} = f(\bm{x}_k, \bm{u}_k)$. SCvx
\cite{Szmuk2025} solves a sequence of convex subproblems, each built from a
local linear approximation of $f$.

At outer iteration $j$, define the deviations
\begin{equation}
  \delta\bm{x}_k = \bm{x}_k - \bm{x}_k^{(j)} \in \R^4,
  \qquad
  \delta\bm{u}_k = \bm{u}_k - \bm{u}_k^{(j)} \in \R^2,
\end{equation}
and linearise $f$ around the current trajectory estimate:
\begin{equation}
  \label{eq:scvx_lin}
  f(\bm{x}_k, \bm{u}_k)
  \approx f(\bm{x}_k^{(j)}, \bm{u}_k^{(j)})
  + A_k^{(j)}\,\delta\bm{x}_k
  + B_k^{(j)}\,\delta\bm{u}_k,
\end{equation}
where
$A_k^{(j)} = \partial f/\partial\bm{x}\big|_{(\bm{x}_k^{(j)},\bm{u}_k^{(j)})}
\in \R^{4\times4}$ and
$B_k^{(j)} = \partial f/\partial\bm{u}\big|_{(\bm{x}_k^{(j)},\bm{u}_k^{(j)})}
\in \R^{4\times2}$ are the Jacobians of $f$ at the current guess. A trust-region
penalty
\begin{equation}
  \eta\bigl(\norm{\delta\bm{x}_k}^2 + \norm{\delta\bm{u}_k}^2\bigr),
  \qquad \eta > 0,
\end{equation}
is added to the subproblem to prevent the solution from moving too far from
the linearisation point. The outer loop proceeds as follows:
\begin{enumerate}[label=\roman*)]
  \item linearise $f$ around the current trajectory to obtain
        $A_k^{(j)}$, $B_k^{(j)}$;
  \item solve the convex subproblem with linearised dynamics \eqref{eq:scvx_lin}
        and trust-region penalty;
  \item set the solution as the new trajectory estimate;
  \item update $\eta$ based on the ratio of actual to predicted cost reduction.
\end{enumerate}

When $f(\bm{x},\bm{u}) = A_d\bm{x}+B_d\bm{u}$ (the CWH case), the Jacobians
are constant and exact, so SCvx converges in a single outer iteration and gives
exactly the same solution as Linear MPC. Its benefit appears only when $f$ is
genuinely nonlinear. SCvx was originally developed for planetary
powered-descent guidance \cite{SzmukAcikmese2018}, before being extended
to rendezvous problems. Szmuk, Acikmese and Reynolds \cite{Szmuk2025}
applied SCvx to rendezvous on near-rectilinear halo orbits, while
Cavenago et al.\ \cite{Cavenago2019} used sequential convex programming for
multi-mode low-thrust spacecraft trajectory optimisation.

More broadly, SCvx belongs to a wider family of sequential convexification and
computationally tractable robust nonlinear MPC methods that seek to preserve the
structure of convex optimisation while handling genuinely nonlinear dynamics and
constraints. A related line of work formulates robust nonlinear MPC via
difference-of-convex decompositions and sequential convex programming
\cite{DoffSotta2026}, conceptually close to the successive convexification
described above.
\section{Numerical Solvers}

\subsection{Overview}

All four MPC families rely on the same predictive-control model but use two
different low-level solvers:
\begin{enumerate}[label=\roman*)]
  \item ADMM with Riccati recursion, for Linear MPC and Tube MPC;
  \item Nesterov accelerated projected gradient, for Fast MPC and the inner
        loop of SCvx.
\end{enumerate}

This separation reflects a broader distinction in numerical MPC between
operator-splitting and augmented-Lagrangian methods on the one hand, and
first-order accelerated gradient methods on the other
\cite{Boyd2011,Richter2012,Rawlings2017}. Classical interior-point methods
remain an important benchmark for medium-scale MPC problems
\cite{Rao1998}, but the present work focuses on methods better aligned with
embedded implementation and repeated online solution.

\begin{table}[H]
\centering
\caption{Numerical solver used for each MPC family.}
\vspace{2mm}
\renewcommand{\arraystretch}{1.4}
\small
\begin{tabular}{p{2.7cm}p{4.7cm}p{4.7cm}p{2.8cm}}
\hline
\textbf{MPC family} & \textbf{Problem solved online}
  & \textbf{Numerical solver} & \textbf{Main idea} \\
\hline
Linear MPC
  & QP with linear dynamics and input bounds
  & ADMM, with the quadratic subproblem solved by backward Riccati recursion
  & Split the quadratic minimisation and the bound enforcement \\
\hline
Tube MPC
  & Tightened nominal QP, plus the real-time correction
    $\bm{u}_k = \bm{v}_k - K\bm{e}_k$
  & Same ADMM solver, applied to the tightened nominal problem
  & Same core as Linear MPC, robustness added by tightening \\
\hline
Fast MPC
  & Condensed QP in the stacked control vector $\bm{v}$
  & Nesterov accelerated projected gradient
  & Matrix-vector products and box projection only \\
\hline
SCvx
  & Sequence of convexified subproblems
  & Outer SCvx loop; each subproblem solved by the Nesterov solver
  & Linearise, solve, update, and repeat \\
\hline
\end{tabular}
\end{table}

\subsection{ADMM for Linear MPC and Tube MPC}

Consider the condensed QP solved at one MPC step:
\begin{equation}
  \label{eq:qp_admm}
  \min_{\bm{v}}\;\tfrac{1}{2}\bm{v}^\top H\bm{v} + \bm{f}^\top\bm{v}
  \quad\text{subject to}\quad G\bm{v} \leq \bm{g},
\end{equation}
This problem has two parts: a quadratic cost (easy to minimise with linear
algebra) and a box constraint (easy to enforce by clipping, but inconvenient
to include directly in the quadratic solve). ADMM separates these two tasks
by introducing an auxiliary variable $\bm{z}$ and rewriting the problem as
\begin{equation}
  \min_{\bm{v},\bm{z}}
  \; \frac{1}{2}\bm{v}^\top H\bm{v} + \bm{f}^\top\bm{v} + I_{\mathcal{U}}(\bm{z})
  \quad\text{subject to}\quad \bm{v} - \bm{z} = 0,
\end{equation}
where $I_{\mathcal{U}}(\bm{z}) = 0$ if $\bm{z} \in \mathcal{X}$ and
$+\infty$ otherwise. With a scaled dual variable $\bm{d}$ and penalty
$\rho > 0$, the iterations are
\begin{align}
  \bm{v}^{(m+1)}
  &= \arg\min_{\bm{v}}
  \!\left(
  \frac{1}{2}\bm{v}^\top H\bm{v} + \bm{f}^\top\bm{v}
  + \frac{\rho}{2}\norm{\bm{v} - \bm{z}^{(m)} + \bm{d}^{(m)}}^2
  \right), \label{eq:admm_v} \\
  \bm{z}^{(m+1)}
  &= \Pi_{\mathcal{U}}\!\left(\bm{v}^{(m+1)} + \bm{d}^{(m)}\right),
  \label{eq:admm_z} \\
  \bm{d}^{(m+1)}
  &= \bm{d}^{(m)} + \bm{v}^{(m+1)} - \bm{z}^{(m+1)}. \label{eq:admm_d}
\end{align}

The $\bm{v}$-update \eqref{eq:admm_v} is a quadratic minimisation
in $\bm{v}$ with no constraints. Since $H \succ 0$ and $\rho > 0$,
the objective is strictly convex and the minimum is unique. Setting
the gradient to zero:
\begin{equation}
  H\bm{v} + \bm{f}
  + \rho\bigl(\bm{v} - \bm{z}^{(m)} + \bm{d}^{(m)}\bigr) = 0.
\end{equation}
Rearranging gives the linear system:
\begin{equation}
  \label{eq:v_update}
  (H + \rho I)\,\bm{v}^{(m+1)}
  = -\bm{f} + \rho\bigl(\bm{z}^{(m)} - \bm{d}^{(m)}\bigr),
\end{equation}
where $H + \rho I \in \R^{2N \times 2N}$ is symmetric positive
definite and therefore invertible. The matrix $H + \rho I$ does
not depend on the current state $\hat{\bm{x}}_k$ or on the ADMM
iteration $m$: it is fixed for the entire mission and can be
factorised once offline (e.g.\ by Cholesky decomposition).
At each ADMM iteration, only the right-hand side
$-\bm{f} + \rho(\bm{z}^{(m)} - \bm{d}^{(m)})$ changes,
so \eqref{eq:v_update} is solved online by a single
triangular back-substitution, at cost $O((2N)^2)$. For large horizons $N$, a more efficient alternative is to exploit the block structure of $H$ and solve \eqref{eq:v_update} via a backward Riccati recursion, which reduces the cost to $O(N)$ operations of size $2 \times 2$ only \cite{Rawlings2017}.
\vspace{3mm}
\begin{remark}
\label{rem:admm_rho_choice}
The choice of $\rho$ is critical. To understand why, recall that
the $\bm{v}$-update reduces to solving the linear system
\eqref{eq:v_update} with matrix $H + \rho I$. This matrix has
eigenvalues $\lambda_i(H) + \rho$, and its condition number is:
\begin{equation}
  \kappa(H + \rho I)
  = \frac{\lambda_{\max}(H) + \rho}{\lambda_{\min}(H) + \rho}.
\end{equation}
The condition number measures how hard the system is to solve:
the larger it is, the more ADMM iterations are needed to converge.

If $\rho$ is too small, the penalty term $\rho I$ is negligible
compared to $H$, so $H + \rho I \approx H$. The ADMM iterations
are then almost unconstrained: the $\bm{v}$-update minimises the
cost freely, far from the box constraint, while the
$\bm{z}$-update projects back onto the box. The two copies
$\bm{v}$ and $\bm{z}$ disagree strongly at each iteration and
take many steps to converge to a common point.

If $\rho$ is too large, the penalty term $\rho I$ dominates
$H$, so the $\bm{v}$-update essentially minimises
$\frac{\rho}{2}\|\bm{v} - \bm{z}^{(m)} + \bm{d}^{(m)}\|^2$
and ignores the original cost. The two copies $\bm{v}$ and
$\bm{z}$ agree quickly, but they converge to a point that
satisfies the constraint while ignoring the cost: the solution
is feasible but far from optimal.

The right balance is when $\rho$ is of the same order as
$\lambda_{\max}(H)$, which is itself of order $\lambda_{\max}(Q)$
since $H = \mathcal{B}^\top\bar{Q}\mathcal{B} + \bar{R}$ and
$\bar{Q}$ is block-diagonal with $Q$ on the diagonal. At this
scale, the penalty is strong enough to enforce agreement between
$\bm{v}$ and $\bm{z}$, but not so strong that it overwhelms
the cost. In this implementation, we set
$\rho = \lambda_{\max}(Q)$, computed once offline.
\end{remark}

From a broader optimisation viewpoint, this splitting strategy is closely
related to the operator-splitting literature for structured convex programmes
\cite{Boyd2011}. In practical MPC software, closely related approaches underlie
modern QP solvers such as OSQP, which combine robustness, warm-starting and
sparse linear algebra particularly well for repeated online solution
\cite{Stellato2020}. Alternative augmented-Lagrangian implementations based on
coordinate descent have also been proposed for fast MPC
\cite{Wu2021}.

\subsection{Nesterov accelerated projected gradient for Fast MPC}

Fast MPC solves the same condensed QP but replaces ADMM by the Nesterov
accelerated projected-gradient scheme. Each iteration consists of three steps:
\begin{enumerate}[label=\roman*)]
  \item compute the extrapolated point
        $\bm{y}^{(j)} = \bm{v}^{(j)} + \beta_j(\bm{v}^{(j)} - \bm{v}^{(j-1)})$,
        with $\beta_j = (j-1)/(j+2)$;
  \item apply one gradient step:
        $\tilde{\bm{v}} = \bm{y}^{(j)} - \tfrac{1}{L}(H\bm{y}^{(j)} + \bm{f})$;
  \item project onto the box constraints:
        $\bm{v}^{(j+1)} = \Pi_{\mathcal{U}}(\tilde{\bm{v}})$.
\end{enumerate}
Each iteration requires only one matrix-vector product $H\bm{y}^{(j)}$ and
one elementwise clip, with no matrix factorisation. The momentum coefficient
$\beta_j$ has a simple interpretation: if the last two iterates moved in
roughly the same direction, it is reasonable to continue a bit further before
computing the next correction.

SCvx adds an outer loop around the same fast inner solver. At each outer
iteration the nonlinear dynamics are linearised, a convex subproblem is formed
and solved by the Nesterov routine, and the trajectory is updated. SCvx thus
does not introduce a different low-level solver: it adds an outer convexification
loop on top of the same fast optimisation core.

The practical appeal of this family of methods lies in its extremely small
per-iteration cost: one matrix-vector product and one projection. This makes
it attractive for hardware-constrained implementations, especially when the
iteration budget can be certified offline \cite{Richter2009,Richter2012}.
Related implementations and algorithmic refinements for fast MPC are discussed
in \cite{Patrinos2016,Frison2016,Jerez2014}.

\section{Numerical Results}
\label{sec:num_results}
\subsection{Test problem}

\subsubsection*{Problem formulation}

We consider a planar orbital rendezvous in the LVLH frame. The chaser starts
at a relative distance of 15\,km from the target and must reach 1\,km, while
remaining inside a safety corridor of half-width $\varepsilon$ around a
prescribed reference trajectory.

The relative motion follows the CWH equations introduced in Section~2. The
state vector is
\[
\bm{x}_k =
(\delta x_k,\; \delta y_k,\; \delta\dot{x}_k,\; \delta\dot{y}_k)^\top
\in \R^4,
\]
where $\delta x_k$ and $\delta y_k$ are the radial and along-track separations
(in metres), and $\delta\dot{x}_k$, $\delta\dot{y}_k$ are the corresponding
relative velocities (in m/s). The control input $u_k = (u_{x,k}, u_{y,k})^\top \in \mathbb{R}^2$ is the thrust
acceleration vector (in m/s\textsuperscript{2}), left unconstrained in this
study (no bound $u_{max}$ is imposed on the commanded thrust). This choice is
deliberate: removing the actuator saturation isolates the effect of the
ill-conditioning discussed in Remark~\ref{rem:kappa_growth} from the effect of input clipping, so
that the performance gap observed between the four MPC families in
Section~5.2 can be attributed unambiguously to the conditioning of $H$ rather
than to differences in how each algorithm handles a bound that is not there. The corridor constraint
requires the chaser to stay close to the reference $\bm{x}_k^{\rm ref}$ at
every step:
\begin{equation}
  \|\bm{x}_k - \bm{x}_k^{\rm ref}\| \leq \varepsilon.
\end{equation}
Disturbances $\bm{\xi}_k \in \mathcal{W}$ model atmospheric drag and actuation
errors. Position and velocity components are bounded separately:
$|\xi_{k,\mathrm{pos}}| \leq \bar{w}_{\mathrm{pos}}$ and
$|\xi_{k,\mathrm{vel}}| \leq \bar{w}_{\mathrm{vel}}$ at each step, with
$\bar{w}_{\mathrm{pos}} = 10$\,m and $\bar{w}_{\mathrm{vel}} = 1$\,m/s.

The continuous-time CWH model is discretised with a zero-order hold and
sampling period $T_s = 10$\,s. This gives the discrete system
\begin{equation}
  \bm{x}_{k+1} = A_d\,\bm{x}_k + B_d\,\bm{u}_k + \bm{\xi}_k,
\end{equation}
with $A_d = e^{A_c T_s} \in \R^{4\times4}$ and
$B_d = \int_0^{T_s}e^{A_c\tau}B_c\,\mathrm{d}\tau \in \R^{4\times2}$,
computed once offline via \texttt{scipy.linalg.expm} at orbital rate
$n = 1.107\times10^{-3}$\,rad/s (500\,km altitude).

The reference trajectory is a straight line in the along-track direction from
$\delta y_0 = 15{,}000$\,m to $\delta y_f = 1{,}000$\,m at constant velocity,
with zero radial displacement. The chaser starts exactly on the reference:
$\bm{x}_0 = \bm{x}_0^{\rm ref}$.

\subsubsection*{MPC parameters and cost}
The quadratic cost uses
\[
Q = 3\times10^3\,\mathrm{diag}(100,100,1,1) \in \R^{4\times4},
\qquad
R = 100\,I_2 \in \R^{2\times2},
\qquad
P = Q.
\]
The matrix $Q$ penalises position errors 100 times more than velocity errors,
reflecting the priority on spatial accuracy. The matrix $R$ penalises thrust
usage uniformly in both directions. The terminal weight $P$ is set equal to
$Q$ as a practical approximation; in the Tube MPC experiments, $P$ is replaced
by the exact DARE solution \eqref{eq:DARE}.

The prediction horizon is $N = 100$ steps, corresponding to 1000\,s of
look-ahead. This much longer horizon than the original benchmark is
deliberately used as a stress test of the conditioning discussed in
Remark~\ref{rem:kappa_growth}: with $T_s = 10$\,s and $N = 100$, the Hessian satisfies $\kappa(H) \approx 3.2\times10^{11}$ — several orders of magnitude worse
than the $\kappa(H) \approx 10^4$--$10^7$ range considered previously.
This extreme conditioning is expected to strongly penalise the Nesterov-based
Fast MPC solver while leaving the exact Cholesky-based solvers (Linear MPC,
Tube MPC, SCvx) unaffected, as confirmed below.

The safety corridor has half-width $\varepsilon = 100$\,m around the
reference trajectory at every step.

For each MPC family, we report the average computation time per step (ms),
the final tracking error, the maximum tracking error, and a qualitative
estimate of the online memory footprint.

\subsection{Comparison results}

Table~\ref{tab:comparison15km1km} summarises the performance of the four
algorithms on the 15\,km to 1\,km rendezvous benchmark. We discuss the
behaviour of each method in turn, supported by the convergence curves shown
in Figure~\ref{fig:convergence}.

\begin{table}[H]
\centering
\caption{Comparison of the four MPC families on the 15\,km to 1\,km
rendezvous test, with $N=100$, $T_s=10$\,s, $\varepsilon=100$\,m,
$\bar w_{\mathrm{pos}}=10$\,m, $\bar w_{\mathrm{vel}}=1$\,m/s, and no
bound on the thrust command.}
\label{tab:comparison15km1km}
\vspace{4mm}
\renewcommand{\arraystretch}{1.2}
\small
\vspace{5mm}
\begin{tabular}{p{2.4cm} p{2.1cm} p{1.8cm} p{1.8cm}}
\hline
\textbf{Algorithm} & \textbf{Time / step}
  & \textbf{Final error} & \textbf{Max error} \\
\hline
Linear MPC & 2.42\,ms  & 9.357\,m   & 13.705\,m \\
Tube MPC   & 2.53\,ms  & 12.655\,m  & 15.122\,m \\
Fast MPC   & 2.60\,ms  & 662.896\,m & 662.896\,m \\
SCvx       & 7.29\,ms  & 11.709\,m  & 14.297\,m \\
\hline
\end{tabular}
\end{table}

\begin{figure}[H]
\centering
\includegraphics[width=0.85\textwidth]{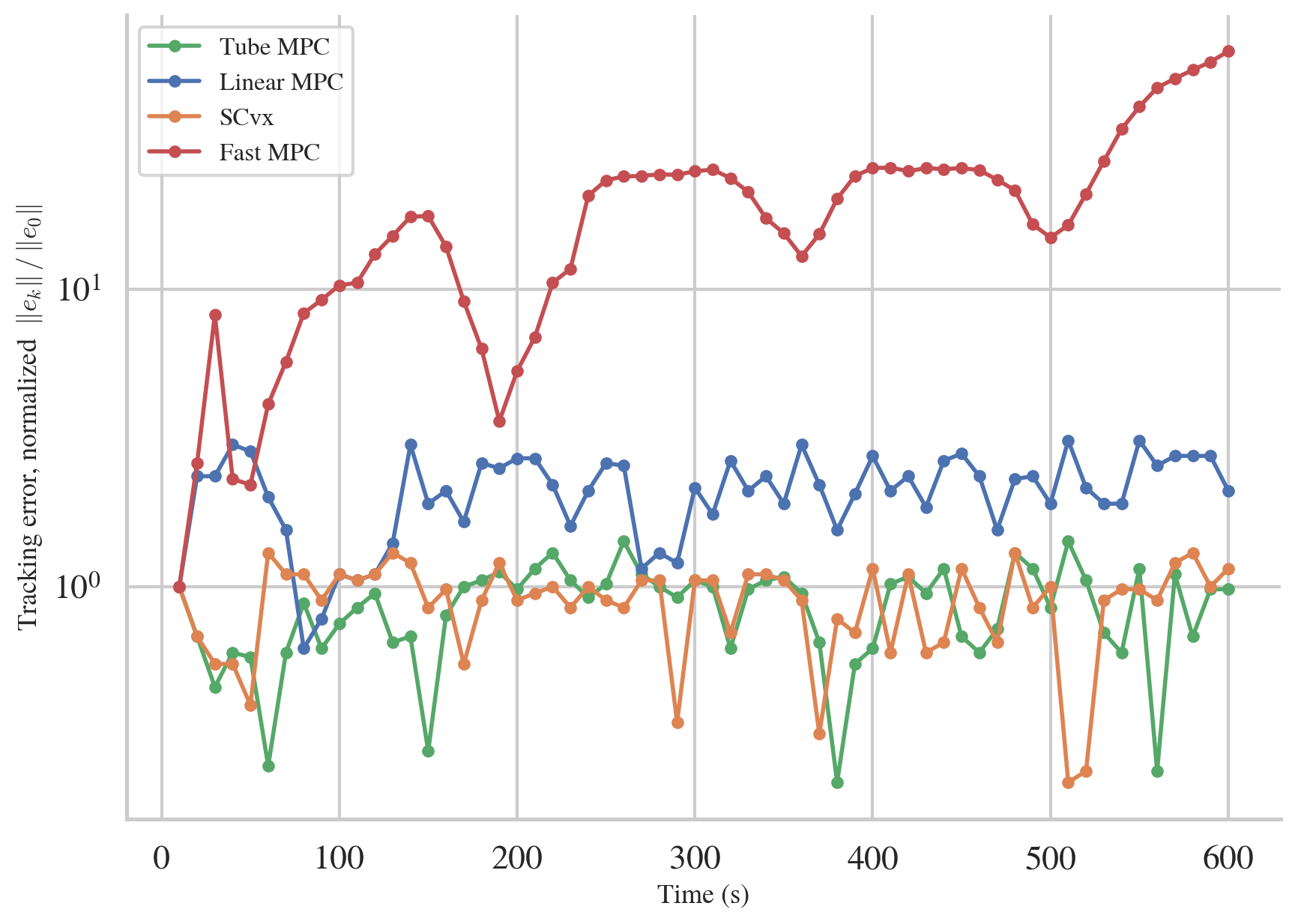}
\caption{Tracking error normalised by the initial value $e_0$, on a log
scale, over the 60-step simulation ($T_s=10$\,s), for the four MPC
families under the new stress-test parameters ($N=100$,
$\varepsilon=100$\,m, $\bar w_{\mathrm{pos}}=10$\,m,
$\bar w_{\mathrm{vel}}=1$\,m/s, no thrust bound).}
\label{fig:convergence}
\end{figure}

\textbf{Linear MPC, Tube MPC and SCvx} all achieve comparable, low tracking
errors (9.4\,m, 12.7\,m and 11.7\,m respectively), well inside the 100\,m
corridor. Their computation time per step remains low (2.4--7.3\,ms), since
all three rely on an exact linear solve (Cholesky factorisation of the
condensed Hessian) that is insensitive to the conditioning of $H$. Contrary
to the original benchmark, Tube MPC does not markedly outperform Linear
MPC here: without a bound on the thrust command, the LQR correction
$K\bm{e}_k$ is applied without saturation, which removes part of the
distinctive advantage that the tightened-corridor formulation had under a
bounded actuator.

\textbf{Fast MPC} collapses under this stress test, with a final and
maximum error of 662.9\,m — more than six times the corridor half-width.
As Figure~\ref{fig:convergence} shows, its normalised error grows by more
than two orders of magnitude instead of decaying, in sharp contrast with
the other three families, which remain close to their initial error
throughout the simulation.

\vspace{2mm}
\begin{remark}[Why does Fast MPC diverge]
\label{rem:fast_mpc_diverge}
The CWH dynamics are marginally stable, meaning the eigenvalues of $A_d$
have modulus close to 1 and therefore never decay. As the prediction
horizon $N$ grows, the condensed Hessian $H$ accumulates more and more
of these non-decaying terms. This pushes $\lambda_{\max}(H)$ upward,
while $\lambda_{\min}(H)$, which is essentially set by the control
weight $R$, stays nearly constant. As a result, the condition number
$\kappa(H) = \lambda_{\max}(H)/\lambda_{\min}(H)$ grows sharply with
$N$, following the same trend already noted in Remark~\ref{rem:kappa_growth}: $\kappa(H) \approx 10^4$ at $N=5$, $\kappa(H) \approx 10^7$ at $N=40$, and here
$\kappa(H) \approx 1.4\times10^{13}$ at $N=100$. A larger condition
number slows the convergence rate of the Nesterov solver, so its fixed
budget of 100 iterations is no longer sufficient to reach an accurate
solution. Reducing the sampling period $T_s$ helps mitigate this
effect, since it limits how much the dynamics evolve at each step and
therefore slows the growth of $\kappa(H)$ with $N$. At $N=100$,
lowering $T_s$ from 100\,s to 10\,s brings the condition number down to
$\kappa(H) \approx 3.2\times10^{11}$, and Fast MPC's tracking error
improves accordingly, though it remains far from acceptable.
\end{remark}

\section{Standard Orbital Rendezvous Manoeuvres}

\subsection{Tube MPC recall and problem formulation}
\label{sec:tube_mpc_recall}

We briefly recall the Tube MPC framework introduced in Section~3.3, now
applied to a three-dimensional CWH model. The state vector is
\[
x_k = (\delta x_k,\; \delta y_k,\; \delta z_k,\;
        \delta\dot{x}_k,\; \delta\dot{y}_k,\; \delta\dot{z}_k)^\top
\in \R^6,
\]
where $\delta x_k$, $\delta y_k$, $\delta z_k$ are the radial, along-track
and cross-track separations in metres, and $\delta\dot{x}_k$,
$\delta\dot{y}_k$, $\delta\dot{z}_k$ are the corresponding relative
velocities in m/s. The control input is the thrust acceleration
$u_k = (u_x, u_y, u_z)^\top \in \R^3$ (m/s$^2$), left \textbf{unconstrained}
in this study. The discrete-time dynamics follow the three-dimensional CWH
model \cite{Clohessy1960,Fehse2003}:
\begin{equation}
\label{eq:dyn3d}
x_{k+1} = A_d x_k + B_d u_k + \xi_k, \qquad \xi_k \sim \mathcal{U}(\mathcal{W}),
\end{equation}
where $A_d = e^{A_c T_s} \in \R^{6\times6}$ is the discrete state-transition
matrix (how the satellite evolves naturally in one step without thrust),
$B_d = \int_0^{T_s} e^{A_c\tau} B_c\,\mathrm{d}\tau \in \R^{6\times3}$
is the input matrix (the effect of a constant thrust pulse on the state),
and $\xi_k$ is the disturbance at step $k$, modelling atmospheric drag,
solar radiation pressure, and actuation errors. The disturbances are drawn
independently and uniformly from the bounded set
$\mathcal{W} = [-\bar{w}, \bar{w}]^6$ with
\[
\bar{w} = (10,\; 10,\; 10,\; 1,\; 1,\; 1)^\top,
\]
where the first three components (in m) bound the position perturbations
and the last three (in m/s) bound the velocity perturbations. Both matrices
$A_d$ and $B_d$ are computed once offline via matrix exponentiation and
loaded into the onboard computer before the mission.

At each step $k$, the nominal MPC solves the following quadratic programme
\cite{Mayne2000,Rawlings2017}:
\begin{equation}
\label{eq:qp3d}
\min_{v_0, \ldots, v_{N-1}} \sum_{j=1}^{N}
(x_{k+j} - x_{k+j}^{\rm ref})^\top Q\, (x_{k+j} - x_{k+j}^{\rm ref})
+ v_j^\top R\, v_j
\end{equation}
subject to the nominal dynamics $x_{k+j+1} = A_d x_{k+j} + B_d v_j$, with
\textbf{no bound imposed on $v_j$}, where:
\begin{itemize}
  \item $Q = 3\times10^3\,\mathrm{diag}(100,100,100,1,1,1) \in \R^{6\times6}$
        is the state weight matrix, penalising position errors 100 times
        more than velocity errors, reflecting the priority on spatial accuracy;
  \item $R = 100\,I_3 \in \R^{3\times3}$ is the control weight matrix,
        penalising thrust usage uniformly in all three directions;
  \item $N = 100$ is the prediction horizon (1000\,s of look-ahead);
  \item $T_s = 10$\,s is the sampling period;
  \item $n = 1.107\times10^{-3}$\,rad/s is the mean orbital rate
        at 500\,km altitude.
\end{itemize}
The terminal weight $P \succeq 0$ is the solution of the Discrete
Algebraic Riccati Equation (DARE) associated with $(A_d, B_d, Q, R)$,
and the LQR gain
$K = (R + B_d^\top P B_d)^{-1} B_d^\top P A_d \in \R^{3\times6}$
is computed once offline \cite{Mayne2005,Rawlings2017}. The applied
control is
\begin{equation}
u_k = v_k^\star - K e_k, \qquad e_k = x_k - x_k^{\rm ref},
\end{equation}
where $v_k^\star$ is the first element of the optimal sequence from
\eqref{eq:qp3d} and $-K e_k$ is the real-time LQR correction that
actively counteracts the disturbance $\xi_k$ between MPC steps.

Two solvers are compared throughout this section \cite{Boyd2011,Richter2012}.
\textbf{ADMM} solves the condensed QP via an offline Cholesky factorisation
of the (unconstrained, since no thrust bound is imposed) Hessian $H$; the
absence of an input bound reduces the usual ADMM iteration to a single
triangular back-substitution per step, which remains immune to the
ill-conditioning of $H$. \textbf{Nesterov} uses the accelerated projected
gradient with step size $1/\lambda_{\max}(H)$, momentum
$\beta_j = (j-1)/(j+2)$, and 100 iterations per MPC step
\cite{Richter2012,Nesterov1983}. All simulations use 30 independent Monte
Carlo runs with disturbances drawn uniformly from $\mathcal{W}$ with
independent seeds.

\subsection{Standard rendezvous manoeuvres}

Orbital proximity operations are structured as a sequence of standardised
approach phases, each imposing a specific reference trajectory
$\{x_k^{\rm ref}\}$ in the LVLH frame \cite{Fehse2003,Breger2008,Hartley2015}.
The key claim studied here is that a single MPC controller with fixed
$Q$, $R$, $K$, $N$ can track the full sequence without any phase-specific
tuning, demonstrating the genericity of the approach
\cite{DiCairano2012,Specht2023,Zamblera2023}. Five standard manoeuvre
types are considered, covering the main proximity operation geometries
used in practice and in the modern rendezvous literature
\cite{Fehse2003,Gaias2015,Lovell2004,Oestreich2023}.

\subsubsection*{Translation}

A straight-line displacement between two points $p_0 = (p_{0,x}, p_{0,y},
p_{0,z})^\top$ and $p_f = (p_{f,x}, p_{f,y}, p_{f,z})^\top$ at constant
velocity \cite{Fehse2003}. This is the simplest manoeuvre, used as the
initial approach to bring the chaser into the proximity of the target:
\begin{equation}
x_k^{\rm ref} = \begin{pmatrix}
p_{0,x} + (p_{f,x} - p_{0,x})\,k/T \\
p_{0,y} + (p_{f,y} - p_{0,y})\,k/T \\
p_{0,z} + (p_{f,z} - p_{0,z})\,k/T \\
(p_{f,x} - p_{0,x})/(T T_s) \\
(p_{f,y} - p_{0,y})/(T T_s) \\
(p_{f,z} - p_{0,z})/(T T_s)
\end{pmatrix},
\end{equation}
where $T$ is the number of steps for the phase. The first three components
are positions in metres and the last three are the constant velocities in
m/s required to travel from $p_0$ to $p_f$ in exactly $T$ steps.

\subsubsection*{R-bar}

The chaser approaches the target along the radial direction $\delta x$,
from below or above, at constant radial velocity \cite{Fehse2003,Gaias2015}.
This is the standard approach for docking with space stations, as it
minimises collision risk in case of thruster failure. The parameter
$\delta x_0$ (m) is the initial radial offset and $-\delta x_0/(TT_s)$
(m/s) is the constant approach velocity:
\begin{equation}
x_k^{\rm ref} = \bigl(\delta x_0(1 - k/T),\; 0,\; 0,\;
                       -\delta x_0/(TT_s),\; 0,\; 0\bigr)^\top.
\end{equation}

\subsubsection*{V-bar}

The chaser approaches from behind along the along-track direction
$\delta y$, parallel to the orbital velocity of the target
\cite{Breger2008,DiCairano2012}. This is the approach used by cargo
vehicles for final docking. The parameter $\delta y_0$ (m) is the
initial along-track separation:
\begin{equation}
x_k^{\rm ref} = \bigl(0,\; \delta y_0(1 - k/T),\; 0,\;
                       0,\; -\delta y_0/(TT_s),\; 0\bigr)^\top.
\end{equation}

\subsubsection*{NMC (Natural Motion Circumnavigation)}

The chaser orbits the target along the 2:1 elliptical free-drift orbit
given by the CWH free solution \cite{Lovell2004,Hartley2015}. This
trajectory requires theoretically zero fuel because it follows the natural
orbital dynamics without continuous thrust, and is used for inspection of
the target satellite \cite{Oestreich2023,Specht2023}. The parameters are:
$A$ (m) the radial semi-axis, $2A$ the along-track semi-axis (the 2:1
ratio is fixed by the CWH equations and cannot be freely chosen), $B$ (m)
the out-of-plane amplitude, $C$ (m) an along-track offset shifting the
ellipse centre, and $\phi$, $\psi$ (rad) the initial phases determining
where on the ellipse the chaser starts:
\begin{align}
\delta x_k^{\rm ref} &= A \cos(n k T_s + \phi), \notag \\
\delta y_k^{\rm ref} &= -2A \sin(n k T_s + \phi) + C, \notag \\
\delta z_k^{\rm ref} &= B \sin(n k T_s + \psi), \notag \\
\delta\dot{x}_k^{\rm ref} &= -An \sin(n k T_s + \phi), \notag \\
\delta\dot{y}_k^{\rm ref} &= -2An \cos(n k T_s + \phi), \notag \\
\delta\dot{z}_k^{\rm ref} &= Bn \cos(n k T_s + \psi).
\end{align}
The velocity components are obtained by analytically differentiating the
position expressions. The NMC was used by Specht et al.\ \cite{Specht2023}
as the nominal trajectory for Tube MPC rendezvous with tumbling targets,
and by Oestreich et al.\ \cite{Oestreich2023} for adaptive tube MPC.

\subsubsection*{Corkscrew}

A variant of the NMC in which the ellipse radius decreases linearly,
causing the chaser to spiral towards the target
\cite{Lovell2004,Zamblera2023}. The parameters $A_0$ (m), $B_0$ (m) and
$C_0$ (m) are the initial values of the radial semi-axis, out-of-plane
amplitude and along-track offset respectively:
\begin{equation}
A(k) = A_0(1 - k/T), \quad B(k) = B_0(1 - k/T),
\quad C(k) = C_0(1 - k/T).
\end{equation}
The position and velocity expressions are the same as the NMC with $A$,
$B$, $C$ replaced by $A(k)$, $B(k)$, $C(k)$. The chaser spirals in three
dimensions and arrives at the origin at $k = T$. This approach provides
a smooth, progressive closing while maintaining the circumnavigation
geometry throughout the approach.

\subsection{Simulation results}

Figure~\ref{fig:tracking_errors} shows the tracking error norm
$\|x_k - x_k^{\rm ref}\|$ as a function of time for the five manoeuvres,
comparing ADMM (blue) and Nesterov (orange) under 30 Monte Carlo runs
with bounded random disturbances $\xi_k \sim \mathcal{U}(\mathcal{W})$,
under the new stress-test parameters ($N=100$, $T_s=10$\,s,
$\bar w_{\mathrm{pos}}=10$\,m, $\bar w_{\mathrm{vel}}=1$\,m/s, no thrust
bound). The thin transparent lines show individual Monte Carlo runs, the
bold line is the mean over 30 runs, and the shaded band represents one
standard deviation.

ADMM achieves lower and more concentrated error trajectories across all
five manoeuvres, settling around 12--13\,m with a comparatively narrow
Monte Carlo spread. Nesterov displays both a higher mean error and a
wider spread on Translation, R-bar and V-bar, and although it is no
longer the closest to ADMM on the NMC manoeuvre as in the original
benchmark, it remains the manoeuvre where the two solvers are closest in
final error (17.5\,m for Nesterov vs 12.6\,m for ADMM), consistent with
the smoother, better-conditioned sub-problem generated by the NMC's
periodic reference at each MPC step.

\begin{figure}[H]
\centering
\includegraphics[width=1\textwidth]{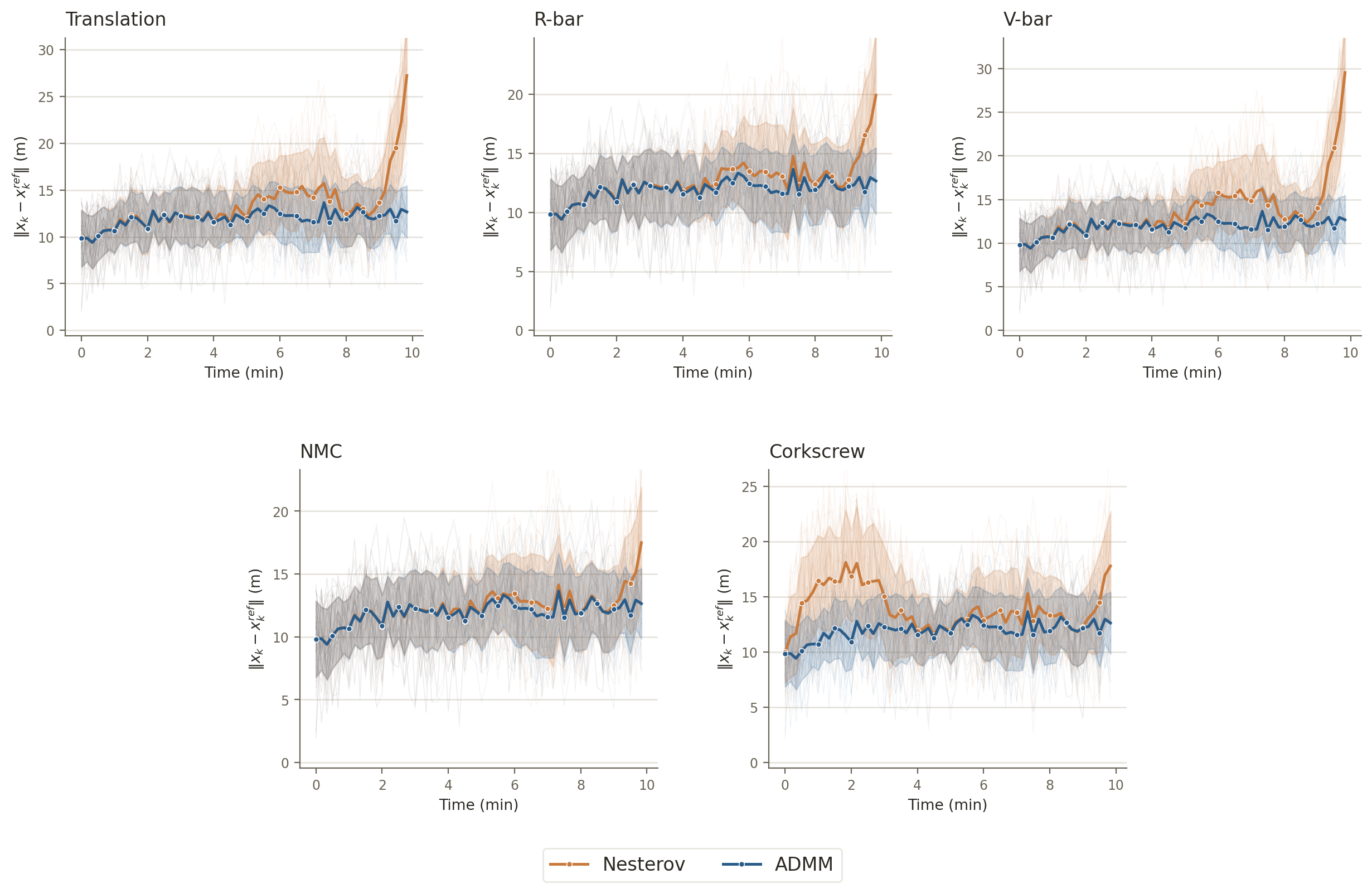}
\caption{Tracking error $\|x_k - x_k^{\rm ref}\|$ as a function of time
for the five standard rendezvous manoeuvres, comparing ADMM (blue) and
Nesterov (orange), under the stress-test parameters ($N=100$,
$T_s=10$\,s, $\bar w_{\mathrm{pos}}=10$\,m, $\bar w_{\mathrm{vel}}=1$\,m/s,
no thrust bound). Thin lines: 30 individual Monte Carlo runs. Bold line:
mean over 30 runs. Shaded band: one standard deviation.}
\label{fig:tracking_errors}
\end{figure}

Table~\ref{tab:manoeuvres} summarises the quantitative results. ADMM
achieves a final tracking error close to 12.6--12.7\,m on all five
manoeuvres, with a standard deviation around 2.8\,m, confirming reliable
performance across Monte Carlo runs. Nesterov is 1.4 to 2.3 times less
accurate than ADMM depending on the manoeuvre, with the gap narrowest on
the NMC (17.5\,m vs 12.6\,m) and widest on V-bar (29.5\,m vs 12.7\,m).

This contrast is again a direct consequence of the ill-conditioning of
the Hessian discussed for the 2D benchmark: with $N=100$ and $T_s=10$\,s,
$\kappa(H)$ reaches the same extreme order of magnitude, so the Nesterov
gradient method converges only partially within its fixed budget of 100
iterations, whereas the ADMM Cholesky factorisation remains immune to
conditioning effects \cite{Boyd2011,Richter2012,Rawlings2017}. Unlike in
the original benchmark, the two solvers now show comparable computation
times per step (6.0--6.3\,ms for Nesterov versus 6.6--7.1\,ms for ADMM):
the much longer horizon ($N=100$ instead of $N=5$) makes even the
Cholesky-based direct solve computationally significant, narrowing the
speed advantage that Nesterov previously held.

\begin{table}[H]
\centering
\caption{Comparison of ADMM and Nesterov on the five standard rendezvous
manoeuvres. Results over 30 Monte Carlo runs with bounded uniform
disturbances $\xi_k \sim \mathcal{U}(\mathcal{W})$,
$\bar{w} = (10,10,10,1,1,1)^\top$, $N=100$, $T_s=10$\,s, no thrust bound.
Bold values indicate the best result in each column.}
\vspace{4mm}
\label{tab:manoeuvres}
\vspace{2mm}
\renewcommand{\arraystretch}{1.25}
\small
\begin{tabular}{llccc}
\hline
\textbf{Manoeuvre} & \textbf{Solver}
  & \textbf{Final error (m)} & \textbf{Max error (m)}
  & \textbf{Time/step (ms)} \\
\hline
\multirow{2}{*}{Translation}
  & ADMM     & $\mathbf{12.662 \pm 2.796}$ & $\mathbf{17.923 \pm 1.612}$ & $7.066$ \\
  & Nesterov & $27.223 \pm 5.292$          & $28.778 \pm 4.049$          & $\mathbf{6.322}$ \\
\multirow{2}{*}{R-bar}
  & ADMM     & $\mathbf{12.655 \pm 2.806}$ & $\mathbf{17.933 \pm 1.602}$ & $6.694$ \\
  & Nesterov & $19.923 \pm 5.284$          & $22.835 \pm 2.516$          & $\mathbf{6.052}$ \\
\multirow{2}{*}{V-bar}
  & ADMM     & $\mathbf{12.667 \pm 2.796}$ & $\mathbf{17.923 \pm 1.613}$ & $6.620$ \\
  & Nesterov & $29.539 \pm 5.381$          & $30.914 \pm 4.264$          & $\mathbf{6.021}$ \\
\multirow{2}{*}{NMC}
  & ADMM     & $\mathbf{12.646 \pm 2.800}$ & $\mathbf{17.925 \pm 1.606}$ & $6.900$ \\
  & Nesterov & $17.496 \pm 4.435$          & $20.898 \pm 2.024$          & $\mathbf{6.274}$ \\
\multirow{2}{*}{Corkscrew}
  & ADMM     & $\mathbf{12.625 \pm 2.797}$ & $\mathbf{17.953 \pm 1.599}$ & $6.908$ \\
  & Nesterov & $17.777 \pm 4.950$          & $23.770 \pm 1.713$          & $\mathbf{6.302}$ \\
\hline
\end{tabular}
\end{table}

These results confirm the key claim: a single MPC controller with fixed
$Q$, $R$, $K$, $N$ successfully tracks all five standard manoeuvre types
without any phase-specific tuning. Under this much longer horizon and
tighter sampling period, ADMM remains the clearly preferable solver when
tracking accuracy is the priority, achieving errors 1.4 to 2.3 times
smaller than Nesterov on every manoeuvre; Nesterov retains only a modest
speed advantage (9--11\% faster per step), which is markedly smaller
than the 2.5$\times$ speed-up observed in the original, shorter-horizon
benchmark.

\section{Reducing the Computational Cost of the Condensed Formulation}
\label{sec:simplification}

\subsection*{Motivation}

The numerical experiments of Sections~5 and~6 are computationally
expensive for large prediction horizons. For $N=100$, building the
condensed matrices
\[
  \mathcal{A} =
  \begin{pmatrix} A_d \\ A_d^2 \\ \vdots \\ A_d^N \end{pmatrix},
  \qquad
  \mathcal{B} =
  \begin{pmatrix}
    B_d          & 0            & \cdots & 0   \\
    A_dB_d       & B_d          & \cdots & 0   \\
    \vdots       & \vdots       & \ddots & \vdots \\
    A_d^{N-1}B_d & A_d^{N-2}B_d & \cdots & B_d
  \end{pmatrix},
  \qquad
  \bar{Q} = \mathrm{blkdiag}(Q,\ldots,Q,P)
\]
requires computing $N$ distinct powers of $A_d$, each obtained from the
matrix exponential $A_d = e^{A_cT_s}$, and storing a dense
$4N\times4N$ matrix $\bar Q$ that literally repeats the same
$4\times4$ block $Q$ a hundred times. Both the memory footprint and the
computation time grow with $N$: for $N=100$, $\bar Q$ alone occupies
$400 \times 400 \times 8$ bytes, that is $1.28$ megabytes, to store
what is in fact a single $4\times4$ matrix repeated 99 times. In this
section we simplify these three objects, using a small parameter
already present in the CWH dynamics, and we check numerically that the
simplification changes nothing in the conclusions of Sections~5 and~6.

\subsection{Simplifying the Condensed Matrices}
\label{sec:simplifying_matrices}

The condensed formulation of Section~3.2 requires computing, for a
horizon of length $N$, all the powers $A_d^{\,m}$ for $m=0,\dots,N-1$,
and stacking them into the block matrix $\mathcal{B}$. For large $N$
this is computationally expensive and, as shown in the previous
remark, numerically delicate because of the ill-conditioning it
induces. We show here that the specific structure of the CWH matrix
$A_c$ allows this computation to be replaced by a much cheaper
polynomial approximation, valid whenever the orbital rate $n$ and the
sampling period $T_s$ satisfy $nT_s \ll 1$, which holds throughout
this work, since $n \approx 1.107\times10^{-3}$\,rad/s and
$T_s \in \{10, 100\}$\,s give $nT_s \approx 0.011$ or $0.111$, both
much smaller than 1.

\subsubsection*{Simplification of $A_d$}

Recall the continuous-time CWH matrix
\[
  A_c =
  \begin{pmatrix}
    0    & 0 & 1    & 0  \\
    0    & 0 & 0    & 1  \\
    3n^2 & 0 & 0    & 2n \\
    0    & 0 & -2n  & 0
  \end{pmatrix}.
\]
Every entry of $A_c$, other than the two structural 1's coming from
the kinematic relation $\dot{\bm x} = \bm v$, is proportional to the
orbital rate $n$ or to $n^2$. With $n \approx 1.107\times10^{-3}$\,rad/s,
the quantity
\[
  \theta = n T_s
\]
is small for both sampling periods used in this work:
$\theta \approx 0.011$ at $T_s=10$\,s and $\theta \approx 0.111$ at
$T_s=100$\,s.

By definition of the matrix exponential,
\[
  A_d = e^{A_cT_s}
  = I + A_cT_s + \frac{(A_cT_s)^2}{2!} + \frac{(A_cT_s)^3}{3!} + \cdots.
\]
Truncating this sum at order $p$ defines the approximation
\[
  A_d^{(p)} = \sum_{k=0}^{p} \frac{(A_cT_s)^k}{k!},
\]
with remainder bounded by
\[
  \bigl\| A_d - A_d^{(p)} \bigr\|
  \leq \sum_{k=p+1}^{\infty} \frac{\|A_cT_s\|^k}{k!}.
\]
Although $\|A_cT_s\|$ itself is not small, its successive powers decay
rapidly: numerically, $\|A_cT_s\| \approx 14.1$, but
$\|(A_cT_s)^2\| \approx 0.31$, $\|(A_cT_s)^3\| \approx 5.1\times10^{-3}$,
and $\|(A_cT_s)^4\| \approx 3.8\times10^{-5}$, so the series converges
quickly despite the first term being large.

Taking $p=3$ and $T_s=10$\,s gives
explicitly
\[
  A_d^{(3)} =
  \begin{pmatrix}
    1.00018 & 0 & 9.99980  & 0.11070 \\
    -1.357\times10^{-6} & 1 & -0.11070 & 9.99918 \\
    3.676\times10^{-5} & 0 & 0.99994 & 0.02214 \\
    -4.070\times10^{-7} & 0 & -0.02214 & 0.99975
  \end{pmatrix}.
\]
Compared with the exact value computed by \texttt{scipy.linalg.expm},
the difference is $\|A_d - A_d^{(3)}\| \approx 1.6\times10^{-6}$, six
orders of magnitude smaller than the position disturbances
$\bar w_{\mathrm{pos}}=10$\,m used throughout this work.

\subsubsection*{Simplification of $B_d$}

The input matrix is defined by the integral
\[
  B_d = \int_0^{T_s} e^{A_c\tau} B_c \, \mathrm{d}\tau.
\]

Expanding the matrix exponential inside the integral as a power series
in $\tau$,
\[
  e^{A_c\tau} = \sum_{k=0}^{\infty} \frac{(A_c\tau)^k}{k!}
  = I + A_c\tau + \frac{(A_c\tau)^2}{2!} + \frac{(A_c\tau)^3}{3!} + \cdots,
\]
and substituting into the integral gives
\[
  B_d = \int_0^{T_s}
  \left(\sum_{k=0}^{\infty} \frac{(A_c\tau)^k}{k!}\right) B_c \, \mathrm{d}\tau
  = \sum_{k=0}^{\infty} \frac{A_c^k B_c}{k!} \int_0^{T_s} \tau^k \, \mathrm{d}\tau,
\]
where the sum and the integral have been exchanged, and $A_c^k$ and
$B_c$ have been taken out of the integral since they do not depend on
$\tau$. Each remaining integral is elementary,
\[
  \int_0^{T_s} \tau^k \, \mathrm{d}\tau = \frac{T_s^{\,k+1}}{k+1},
\]
so that
\[
  B_d = \sum_{k=0}^{\infty} \frac{A_c^k B_c}{k!} \cdot \frac{T_s^{\,k+1}}{k+1}
  = \sum_{k=0}^{\infty} \frac{T_s^{\,k+1}}{(k+1)!} A_c^k B_c,
\]
the same type of series as for $A_d$, but with the extra factor
$T_s/(k+1)$ coming from the integration. Truncating this series at
order $p$ defines the approximation
\[
  B_d^{(p)} = \sum_{k=0}^{p} \frac{T_s^{\,k+1}}{(k+1)!} A_c^k B_c.
\]
With $p=3$ and $T_s=10$\,s, this gives explicitly
\[
  B_d^{(3)} =
  \begin{pmatrix}
    49.999  & 0.369   \\
    -0.369  & 49.998  \\
    9.9998  & 0.1107  \\
    -0.1107 & 9.9992
  \end{pmatrix},
\]
differing from the exact value by $\|B_d - B_d^{(3)}\| \approx
3.2\times10^{-6}$, negligible compared with the entries of $B_d$
itself, which are of order 10 to 50.

\subsubsection*{Simplification of $\bar Q$}

This simplification introduces no error at all. The matrix
$\bar{Q} = \mathrm{blkdiag}(Q,\ldots,Q,P) \in \mathbb{R}^{4N\times4N}$
contains nothing more than 99 literal copies of the $4\times4$ matrix
$Q$, plus one copy of the terminal weight $P$. Let
$\bm e = \bm X - \bm X^{\mathrm{ref}} \in \mathbb{R}^{4N}$ denote the
stacked tracking error, partitioned into its $N$ constituent 4-vectors
\[
  \bm e =
  \begin{pmatrix} \bm e_1 \\ \bm e_2 \\ \vdots \\ \bm e_N \end{pmatrix},
  \qquad \bm e_j = \bm x_{k+j} - \bm x_{k+j}^{\mathrm{ref}} \in \mathbb{R}^4.
\]
Because $\bar{Q}$ is block-diagonal, the product $\bar{Q}\bm e$ acts on
each block independently:
\[
  \bar{Q}\bm e =
  \begin{pmatrix} Q\bm e_1 \\ Q\bm e_2 \\ \vdots \\ P\bm e_N \end{pmatrix}.
\]
Computing the right-hand side by applying the $4\times4$ matrix $Q$
(respectively $P$ for the last block) to each 4-vector $\bm e_j$ gives,
term for term, exactly the same result as multiplying by the full
$400\times400$ matrix $\bar{Q}$, at a fraction of the memory and
computation cost. The same reasoning applies wherever $\bar{Q}$ appears
in the condensed cost, in particular in
$H = \mathcal{B}^\top\bar{Q}\mathcal{B}+\bar{R}$ and
$\bm f = \mathcal{B}^\top\bar{Q}\bm c$.

\subsubsection*{Summary of the Three Simplifications}
\vspace{-3mm}
\begin{table}[H]
\centering
\caption{The three matrix simplifications, what they replace, and the
accuracy loss incurred.}
\renewcommand{\arraystretch}{1.3}
\vspace{4mm}
\begin{tabular}{p{2cm}ccc}
\hline
\textbf{Simplification} & \textbf{Replaces} & \textbf{By} & \textbf{Accuracy loss} \\
\hline
1. $A_d$ truncation & matrix exponential & degree-3 polynomial & $1.6\times10^{-6}$ \\
2. $B_d$ truncation & matrix integral & degree-3 polynomial & $3.2\times10^{-6}$ \\
3. $\bar{Q}$ block-wise & $400\times400$ dense matrix & $4\times4$ matrix, applied $N$ times & exact (zero) \\
\hline
\end{tabular}
\end{table}

\subsection{Numerical Validation Across the Parameter Space}

Simplifications 1 and 2 were applied together, with $p=3$, and
substituted into every experiment of Sections~5 and~6, without any
other change to the algorithms, disturbances, or horizon length. Since
$\mathcal{A}$, $\mathcal{B}$ and $\bar{Q}$ enter the QP solved by both
ADMM and Nesterov identically, this test also checks that neither
solver's conclusions depend on how the dynamics matrices are computed.

\vspace{3mm}
\begin{table}[H]
\centering
\caption{Final tracking error, exact matrices (\texttt{expm}) versus
simplified matrices (truncated series, $p=3$), on the 2D rendezvous
benchmark of Section~5 ($N=100$, $T_s=10$\,s).}
\renewcommand{\arraystretch}{1.3}
\vspace{1cm}
\begin{tabular}{lccc}
\hline
\textbf{Algorithm} & \textbf{Exact (m)} & \textbf{Simplified (m)} & \textbf{Difference} \\
\hline
Linear MPC & $9.356931$   & $9.356931$   & $6.8\times10^{-7}$ \\
Tube MPC   & $12.654944$  & $12.654944$  & $9.1\times10^{-8}$ \\
Fast MPC   & $662.896186$ & $662.896199$ & $1.3\times10^{-5}$ \\
SCvx       & $11.708934$  & $11.708950$  & $1.5\times10^{-5}$ \\
\hline
\end{tabular}
\end{table}

\begin{table}[H]
\centering
\caption{Final tracking error (mean over 30 Monte Carlo runs), exact
versus simplified matrices, on the five standard manoeuvres of
Section~6 ($N=100$, $T_s=10$\,s).}
\renewcommand{\arraystretch}{1.3}
\vspace{5mm}
\small
\begin{tabular}{|l|l|c|c|c|}
\hline
\textbf{Manoeuvre} & \textbf{Solver} & \textbf{Exact (m)} & \textbf{Simplified (m)} & \textbf{Difference} \\
\hline
\multirow{2}{*}{Translation} & ADMM     & $12.662384$ & $12.662385$ & $5.7\times10^{-7}$ \\
                             & Nesterov & $27.223433$ & $27.223434$ & $7.8\times10^{-7}$ \\
\hline
\multirow{2}{*}{R-bar}       & ADMM     & $12.654630$ & $12.654631$ & $6.6\times10^{-7}$ \\
                             & Nesterov & $19.922696$ & $19.922696$ & $3.0\times10^{-7}$ \\
\hline
\multirow{2}{*}{V-bar}       & ADMM     & $12.667129$ & $12.667130$ & $5.6\times10^{-7}$ \\
                             & Nesterov & $29.538638$ & $29.538639$ & $7.9\times10^{-7}$ \\
\hline
\multirow{2}{*}{NMC}         & ADMM     & $12.645880$ & $12.645880$ & $6.1\times10^{-7}$ \\
                             & Nesterov & $17.496218$ & $17.496218$ & $5.8\times10^{-7}$ \\
\hline
\multirow{2}{*}{Corkscrew}   & ADMM     & $12.624788$ & $12.624789$ & $4.9\times10^{-7}$ \\
                             & Nesterov & $17.777034$ & $17.777034$ & $2.5\times10^{-7}$ \\
\hline
\end{tabular}
\end{table}

\subsection{Cost and Validity of the Approximation
}

Three things make this simplification worth adopting. First, it is
accurate: the largest difference observed between exact and simplified
matrices, across every algorithm and every manoeuvre tested, is
$1.5\times10^{-5}$\,m, at least six orders of magnitude smaller than the
tracking errors themselves, which range from 9\,m to 663\,m. Second, it
preserves the conditioning of the problem exactly: the condition number
of the Hessian is unchanged to four significant figures,
$\kappa(H) \approx 3.216\times10^{11}$ in both cases, so the growth of
$\kappa(H)$ with $N$ discussed in Remark~\ref{rem:kappa_growth}, and the
resulting divergence of Fast MPC, are confirmed to be properties of the
underlying CWH dynamics itself, not artefacts of how $A_d$ and $B_d$
happen to be computed. Third, it is cheaper: the exact route requires
one matrix exponential evaluation and $N-1$ successive matrix
multiplications to build $\mathcal{A}$ and $\mathcal{B}$, plus a
$400\times400$ matrix held in memory for $\bar Q$; the simplified route
requires only a fixed polynomial of degree 3 in $A_c$, evaluated once,
and a $4\times4$ matrix $Q$ applied block by block, reducing the memory
needed for $\bar Q$ from $1.28$ megabytes to $128$ bytes for $N=100$.
Taken together, these three points show that the simplification can
replace the exact matrix exponential and the dense block-diagonal cost
matrix at no practical cost in accuracy, for either solver.

\subsection{The Fuel–Accuracy Trade-off Through $Q$ and $R$}

Throughout Sections~5 and~6, the weight matrices $Q$ and $R$ were kept
fixed. It is natural to ask what happens when they are allowed to vary:
does pushing $Q$ higher relative to $R$ genuinely buy extra tracking
accuracy, and at what fuel cost? To investigate this, we ran the Tube
MPC controller on the same 15\,km to 1\,km benchmark under five
settings, ranging from a fuel-priority regime ($R \gg Q$) to a
precision-priority regime ($Q \gg R$), all other parameters (horizon,
disturbance, sampling period) held fixed across the five runs.

\begin{figure}[H]
\centering
\includegraphics[width=0.95\textwidth]{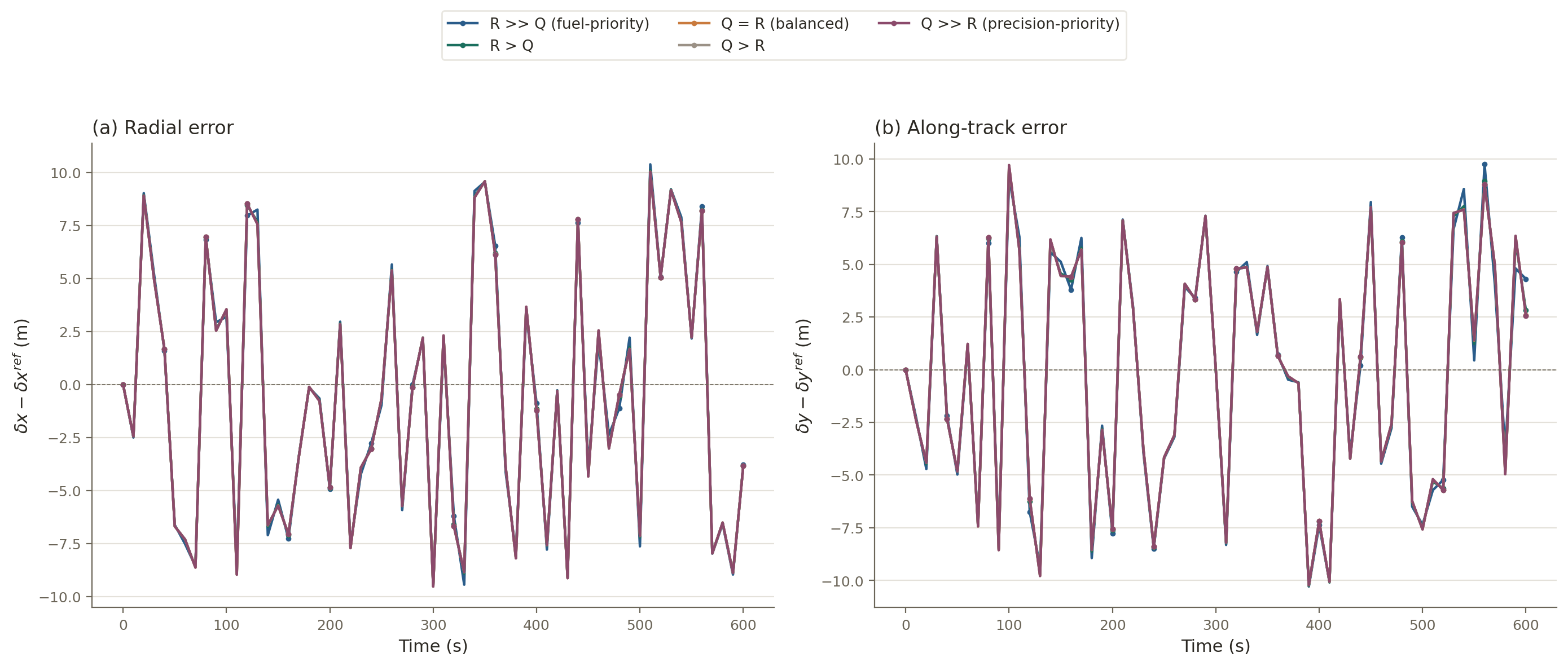}
\caption{Radial (a) and along-track (b) tracking error over time, for
the five $Q/R$ settings tested. The five curves are nearly
indistinguishable in both panels.}
\label{fig:qr_trajectories}
\end{figure}

The first thing to notice, in Figure~\ref{fig:qr_trajectories}, is
that the five trajectories are almost impossible to tell apart. Whether
$Q$ dominates $R$ by a factor of a thousand or the reverse, the radial
and along-track error curves follow essentially the same path
throughout the simulation. This was not the expected outcome: the
usual intuition is that weighting position more heavily should pull
the trajectory closer to the reference.

\begin{figure}[H]
\centering
\includegraphics[width=0.75\textwidth]{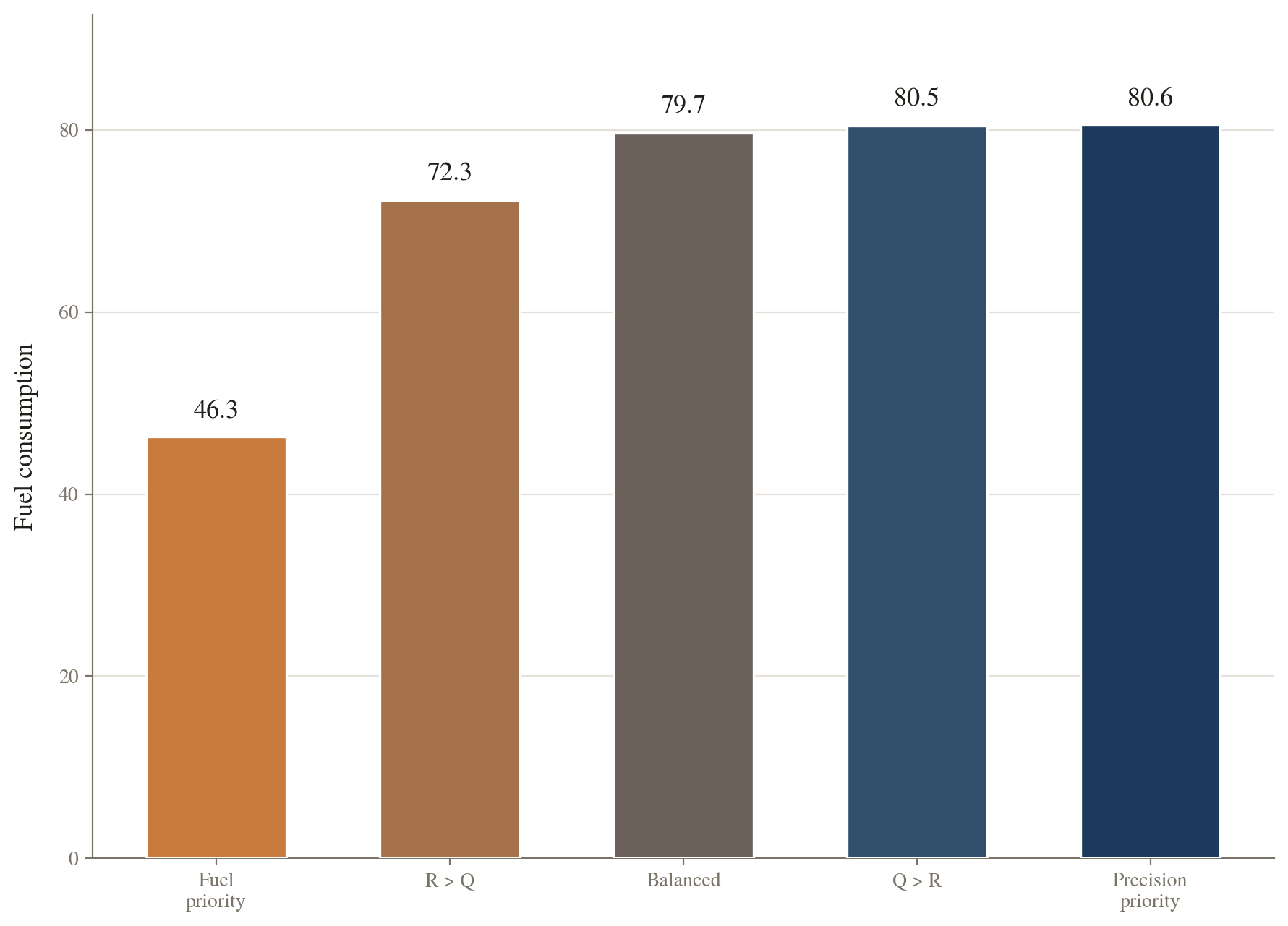}
\caption{Total fuel consumption $\sum_k \|u_k\|^2$ for the same five
$Q/R$ settings. Unlike the tracking error, fuel consumption responds
clearly to the weighting, nearly doubling from the fuel-priority to
the precision-priority end of the range.}
\label{fig:qr_fuel}
\end{figure}

Fuel consumption tells a different story. As shown in
Figure~\ref{fig:qr_fuel}, it rises steadily as $Q/R$ increases, from
roughly 46 to 81 (arbitrary units), nearly doubling across the range
tested. So the weighting is not inert: it clearly costs more fuel to
push $Q$ higher, without buying the accuracy one might expect in
return, a conclusion Figure~\ref{fig:qr_trajectories} already suggested
and Figure~\ref{fig:two_sided_bound_validation} confirms from a different angle:
the correction factor $\alpha$ settles into a proved interval once
$\rho$ is known, so beyond a certain point no further increase in $Q$
can improve the guaranteed error bound.

The explanation lies in how Tube MPC splits the control effort. The
applied thrust is $u_k = v_k - Ke_k$, where $v_k$ is the nominal
command shaped by $Q$ and $R$, and $-Ke_k$ is the real-time LQR
correction reacting to the error $e_k$ already measured onboard.
Comparing the two terms over a representative run shows that
$\|Ke_k\|$ is roughly five to ten times larger than $\|v_k\|$: the
feedback correction does most of the work of keeping the chaser near
the reference, largely independently of how $Q$ and $R$ are set.
Pushing $Q$ higher still increases the nominal command, and therefore
the fuel bill, but the tracking error is already dominated by a
mechanism that $Q$ and $R$ barely influence.
\vspace{3mm}
\begin{remark}
This behaviour is specific to the unconstrained regime studied
throughout this benchmark, where no bound is imposed on the thrust
command. Without a saturation limit, the LQR correction $-Ke_k$ is
always free to apply whatever thrust $K$ prescribes, fully compensating
the disturbance regardless of how $v_k$ is tuned; this is precisely why
the correction term dominates so consistently above. Under a bounded
actuator, this would no longer hold: if $\|Ke_k\|$ exceeded $u_{\max}$,
the correction would saturate, and the nominal command $v_k$ would then
have to take over part of the disturbance rejection it plays no role in
here. Weighting $Q$ more heavily could then plausibly buy back some of
the accuracy sacrificed to saturation, so the conclusion reached here
should be read as a property of the unconstrained case specifically,
not of Tube MPC in general.
\end{remark}

For this benchmark, the practical takeaway is straightforward: within
the regime studied here (Tube MPC, no bound on the thrust command),
there is little to gain from weighting $Q$ heavily. A fuel-priority
setting achieves essentially the same tracking performance at a
noticeably lower cost, which is the opposite of what a naive reading
of the cost function $\ell = e^\top Q e + u^\top R u$ might suggest.

\subsection{A Closed-Form Estimate of the Maximum Tracking Error}
\label{sec:zbar_formula}

We seek a simple relation between the weight matrices $Q$, $R$ and the
maximum position error $\bar z$ that the Tube MPC controller can
guarantee, without having to run a full closed-loop simulation for
every candidate $(Q,R)$. The derivation starts from the closed-loop
error dynamics already established in Proposition~1,
\begin{equation}
\label{eq:err_dyn_recall}
e_{k+1} = A_{cl}\,e_k + \xi_k, \qquad A_{cl} = A_d - B_dK, \qquad \|\xi_k\|\le\bar w,
\end{equation}
where $K$ is the LQR gain obtained from $Q$ and $R$ via the DARE
\eqref{eq:DARE} and the gain formula \eqref{eq:K}.

Consider first a one-dimensional version of this recursion, in which
the error is a single number and $A_{cl}$ is replaced by a scalar
$\rho \in (0,1)$: $e_{k+1} = \rho\,e_k + \xi_k$, with $|\xi_k|\le\bar w$
and $e_0=0$. Unrolling this recursion gives
$e_k = \sum_{i=0}^{k-1}\rho^{\,k-1-i}\xi_i$, exactly as for the full
error dynamics derived earlier. The worst case occurs when every
perturbation takes its maximal value $\bar w$, aligned in the same
direction; taking absolute values and letting $k\to\infty$,
\[
|e_k| \le \bar w\sum_{i=0}^{k-1}\rho^i
\xrightarrow[k\to\infty]{}
\bar w\sum_{i=0}^{\infty}\rho^i = \frac{\bar w}{1-\rho},
\]
using the geometric series identity, valid since $0<\rho<1$. This bound
is attained in the limit by choosing $\xi_i=\bar w$ for every $i$, so
for the scalar case the worst-case error radius is exactly
$\bar z_{\text{scalar}} = \bar w/(1-\rho)$.

The true error dynamics \eqref{eq:err_dyn_recall} are not scalar:
$e_k\in\mathbb{R}^n$ ($n=4$ for the planar rendezvous of Section~5,
$n=6$ for the three-dimensional manoeuvres of Section~6) and
$A_{cl}\in\mathbb{R}^{n\times n}$. The scalar identity above still
governs this vector case through the eigenvalues of $A_{cl}$. A scalar
$\lambda$ and a nonzero vector $v$ form an eigenpair of $A_{cl}$ if
$A_{cl}v=\lambda v$: along the direction $v$, the action of the full
matrix reduces to multiplication by the single number $\lambda$, with
no mixing into other directions. Intuitively, an eigenvalue is the
factor by which the error is scaled along one particular direction of
the state space at each step: a value close to $1$ means the error
barely shrinks along that direction, while a value close to $0$ means
it is corrected almost immediately. Decomposing the error along a
basis of eigenvectors $v_1,\ldots,v_n$ with eigenvalues
$\lambda_1,\ldots,\lambda_n$, each coefficient evolves independently of
the others, exactly according to the scalar recursion solved above,
with $\rho$ replaced by the corresponding $\lambda_j$.

Since $A_{cl}$ is stable, every $|\lambda_j|<1$, but they need not be
equal. As $k$ grows, the component associated with the largest
eigenvalue in modulus decays the slowest, and eventually dominates
every other mode, whose relative contribution vanishes geometrically.
This is why the worst-case error is governed, on the timescales
relevant to the mission, by a single number: the spectral radius
\begin{equation}
\rho = \max_j \bigl|\lambda_j(A_{cl})\bigr|.
\end{equation}
This reasoning relies on there being a well-separated dominant
eigenvalue: the direction associated with $\rho$ must decay
sufficiently slower than every other direction for its contribution to
eventually outweigh theirs. If several eigenvalues share a modulus
close to $\rho$, no single direction dominates within the horizon of
interest, and reducing the error dynamics to the scalar quantity $\rho$
is no longer justified; the estimate below should then be expected to
lose accuracy, exactly as observed in Remark~\ref{rem:zbar_breakdown}
for small ratios $Q/R$. Under the assumption that $\rho$ is well
separated from the rest of the spectrum, the estimate reads
\begin{equation}
\bar z \approx \frac{\bar w}{1-\rho},
\end{equation}
with $\rho$ now computed from the true $n\times n$ matrix $A_{cl}$
rather than a scalar, expressing $\bar z$ entirely in terms of $Q$ and
$R$ through the chain $(Q,R)\to P\to K\to A_{cl}\to\rho$, using only
quantities already introduced in Section~3.3.

This estimate is not exact for $n>1$, for two reasons the scalar
derivation cannot see: sub-dominant eigenvalues still contribute to
$\bar z$ at finite $k$, even as their relative contribution vanishes
asymptotically; and the eigenvectors of $A_{cl}$ are generally not
orthogonal, so recombining the error along them can inflate the bound
beyond the scalar prediction. Before correcting for these two effects,
it is worth noting that the naive quantity $\|e\|_\infty$ used
implicitly above is not even dimensionally well posed: it mixes
position components (in metres) with velocity components (in metres
per second), so its numerical value depends on an arbitrary choice of
units for the state. We remove this ambiguity by working in the
disturbance-scaled error
\begin{equation}
\tilde e = D^{-1}e, \qquad D = \operatorname{diag}(\bar w),
\end{equation}
so that $\|D^{-1}\xi_k\|_\infty \le 1$ for every admissible disturbance,
and by writing the scaled closed-loop matrix $\tilde A = D^{-1}A_{cl}D$.
The scalar identity of the one-dimensional case above generalises
directly to this scaled variable through its eigendecomposition
$\tilde A = V\Lambda V^{-1}$, and the scaled margin is defined as
\begin{equation}
\label{eq:scaled_margin}
\tilde z \;=\; \max_j \sum_{i\ge0} \sum_l \bigl|[\tilde A^{\,i}]_{jl}\bigr|,
\end{equation}
which is dimensionless by construction, and from which the physical
margin is recovered componentwise by $\bar z_j \le \bar w_j\,\tilde z$.

Rather than absorbing the two effects above into a single calibrated
constant, they can be bounded exactly.

\begin{proposition}[Two-sided bound]
\label{prop:two_sided_bound}
Let $\rho = \rho(A_{cl})$ be the spectral radius of the closed-loop
error matrix and let $\tilde A = V\Lambda V^{-1}$ be its scaled
eigendecomposition. Then
\begin{equation}
\label{eq:two_sided_bound}
\frac{1}{1-\rho} \;\le\; \tilde z \;\le\; \frac{\kappa_\infty(V)}{1-\rho},
\qquad
\kappa_\infty(V) = \|V\|_\infty\,\|V^{-1}\|_\infty.
\end{equation}
\end{proposition}

\begin{proof}
\emph{Lower branch.} The spectral radius of a matrix never exceeds any
of its induced norms, so $\|\tilde A^{\,i}\|_\infty \ge \rho(\tilde
A^{\,i}) = \rho^i$ for every $i$. Summing over $i$,
\[
\tilde z = \sum_{i\ge0} \|\tilde A^{\,i}\|_\infty
\;\ge\; \sum_{i\ge0} \rho^i \;=\; \frac{1}{1-\rho}.
\]

\emph{Upper branch.} Writing $\tilde A^{\,i} = V\Lambda^i V^{-1}$ and
using submultiplicativity of the induced $\infty$-norm,
\[
\|\tilde A^{\,i}\|_\infty
= \|V\Lambda^i V^{-1}\|_\infty
\;\le\; \|V\|_\infty\,\|\Lambda^i\|_\infty\,\|V^{-1}\|_\infty
= \kappa_\infty(V)\,\rho^i,
\]
since $\Lambda$ is diagonal with $\|\Lambda^i\|_\infty = \rho^i$.
Summing the resulting geometric series gives
\[
\tilde z \;\le\; \kappa_\infty(V) \sum_{i\ge0} \rho^i
\;=\; \frac{\kappa_\infty(V)}{1-\rho}. \qedhere
\]
\end{proof}

Writing $\tilde z = \alpha/(1-\rho)$, Proposition~\ref{prop:two_sided_bound}
states that $1 \le \alpha \le \kappa_\infty(V)$: rather than calibrating
a single constant $\alpha$ against simulation data, as an earlier version
of this estimate did, $\alpha$ is shown here to lie inside an interval
computed exactly, on both sides, from $Q$ and $R$ alone, with no
simulation and no fitting. In particular the lower bound $\alpha \ge 1$
is not an empirical observation but a proved consequence of $\rho$
never exceeding an induced norm; any calibration returning $\alpha<1$
would in fact be diagnosing a units inconsistency in the uncorrected
$\|e\|_\infty$, not a property of the dynamics.

Figure~\ref{fig:two_sided_bound_validation} verifies
\eqref{eq:two_sided_bound} on 1000 weight pairs $(Q,R)$ drawn
independently over four orders of magnitude: every point falls inside
the admissible band between the two proved branches, as it must.

\begin{figure}[H]
\centering
\includegraphics[width=0.68\textwidth]{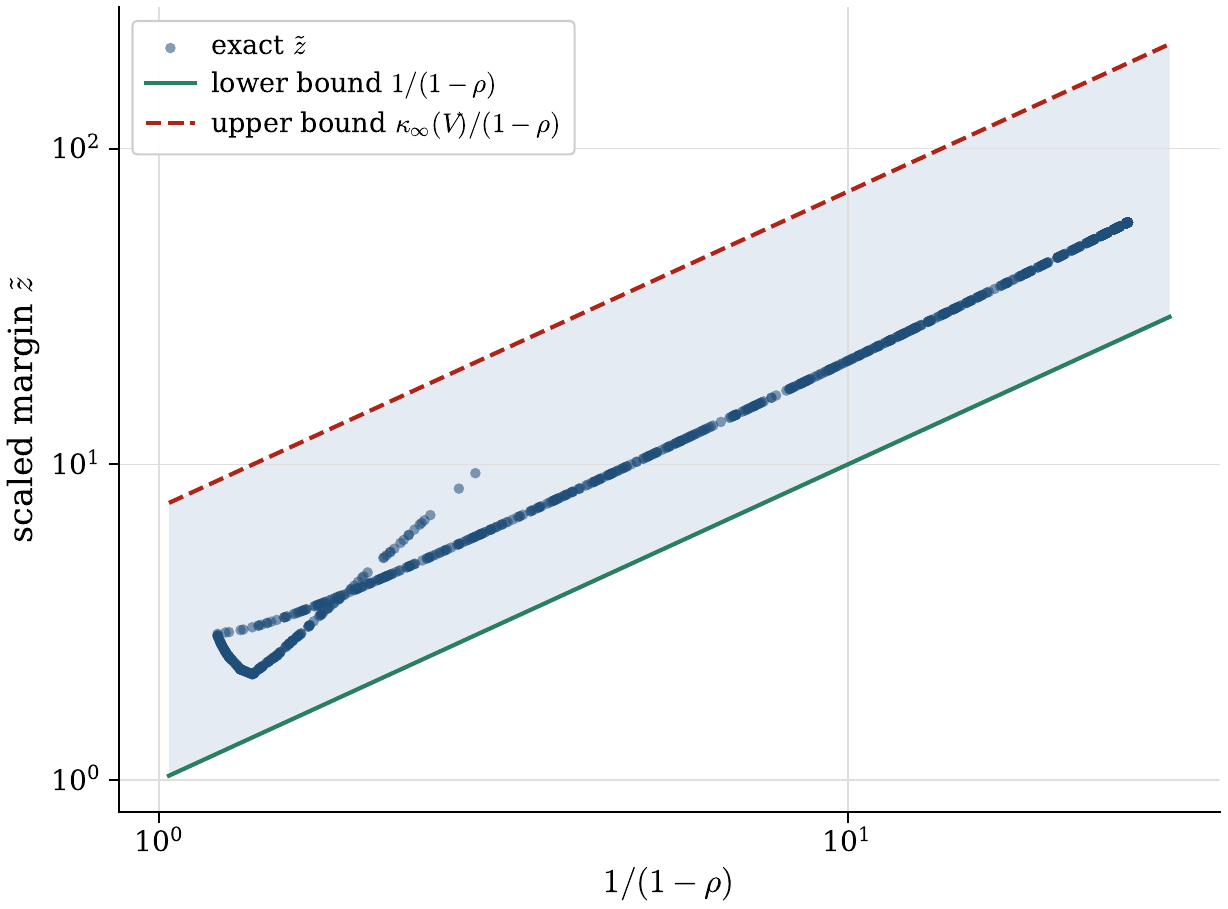}
\caption{The two-sided bound, verified on every weight pair. Each point
is the scaled margin $\tilde z$ computed exactly from the closed-loop
error dynamics, for one of 1000 randomly drawn pairs $(Q,R)$, plotted
against $1/(1-\rho)$. Every point lies between the lower branch
$1/(1-\rho)$ (proved exact lower bound) and the upper branch
$\kappa_\infty(V)/(1-\rho)$ (proved upper bound); no point can lie
outside the shaded band by Proposition~\ref{prop:two_sided_bound}.}
\label{fig:two_sided_bound_validation}
\end{figure}

\begin{remark}[What a single calibrated constant is still good for]
A fixed $\alpha$ calibrated once on this data set predicts $\tilde z$
to a few percent on average, but with a much larger worst-case error,
so it remains at best an approximate surrogate, not a certified one;
given that both exact branches of \eqref{eq:two_sided_bound} are
available at negligible extra cost, there is little reason to prefer
it. Its practical value lies specifically in the lower branch, which
requires only the spectral radius $\rho$ and no eigendecomposition at
all: if $\bar w_{\text{pos}}/(1-\rho) > \varepsilon$, the weighting is
certifiably infeasible and can be rejected immediately, without
computing $\bar z$ exactly. This makes the lower branch a cheap
screening test for the $(Q,R)$ design loop, ahead of the exact
evaluation of $\bar z$. As anticipated, an exact, closed-form value of
$\bar z$ turns out to exist and is developed in full below.
\end{remark}

\subsubsection*{Large-scale validation}

Formula~\eqref{eq:two_sided_bound} was tested against 1000 random
$(Q,R)$ pairs, with $Q$ and $R$ each drawn independently over four
orders of magnitude. For each pair, $\rho$, $V$ and the exact scaled
margin $\tilde z$ were obtained from the true closed-loop dynamics;
unlike a calibrated point estimate, no constant is fitted here, so
there is nothing to validate \emph{out of sample}: the claim under
test is simply that inequality~\eqref{eq:two_sided_bound} holds, which
Proposition~\ref{prop:two_sided_bound} already guarantees for every
$(Q,R)$, including ratios far outside this sample.

Table~\ref{tab:bound_regimes} summarises the results by regime.
\vspace{-1mm}
\begin{table}[H]
\centering
\caption{Position of the exact ratio $\alpha=\tilde z(1-\rho)$ inside
the proved bracket $[1,\kappa_\infty(V)]$, over 1000 random $(Q,R)$
pairs, grouped by weight ratio $Q/R$. Unlike a calibrated point
estimate, the bound is satisfied in every regime without exception,
including $Q/R<0.1$, the regime in which a fixed-$\alpha$ estimate
degrades to a $49.52\%$ worst-case error.}
\vspace{5mm}
\label{tab:bound_regimes}
\renewcommand{\arraystretch}{1.25}
\begin{tabular}{|l|c|c|c|c|}
\hline
\textbf{Range of $Q/R$} & \textbf{Points} & $\boldsymbol{\alpha_{\min}}$ & $\boldsymbol{\alpha_{\max}}$ & \textbf{Violations} \\
\hline
$Q/R < 0.01$          & 167 & 1.56 & 3.14 & 0/167 \\
\hline
$0.01 \le Q/R < 0.1$  & 156 & 1.58 & 2.98 & 0/156 \\
\hline
$0.1 \le Q/R < 1$     & 190 & 1.60 & 3.20 & 0/190 \\
\hline
$1 \le Q/R < 10$      & 171 & 1.59 & 3.55 & 0/171 \\
\hline
$10 \le Q/R < 100$    & 138 & 2.05 & 3.73 & 0/138 \\
\hline
$Q/R \ge 100$         & 178 & 1.23 & 3.73 & 0/178 \\
\hline
\end{tabular}
\end{table}

The contrast with a calibrated point estimate is precisely the point.
Fixing a single $\alpha\approx1.50$ and using it for every ratio
degrades sharply below $Q/R=0.1$, reaching a $49.52\%$ worst-case
error, because a point estimate cannot track a quantity that genuinely
moves with $Q/R$. The bound of Proposition~\ref{prop:two_sided_bound}
does not suffer from this: it does not claim a single value for
$\alpha$, only that $\alpha$ lies in $[1,\kappa_\infty(V)]$, and this
claim is verified without exception across every regime in
Table~\ref{tab:bound_regimes}, including the small-ratio regime that
breaks a fixed-$\alpha$ estimate.

\begin{remark}[Why a fixed-$\alpha$ estimate degrades at small $Q/R$,
and why the bound does not]
\label{rem:zbar_breakdown}
A point estimate of $\alpha$ rests on a single assumption: that one
eigenvalue of $A_{cl}$, the spectral radius $\rho$, dominates all the
others strongly enough that the error dynamics behave, after a few
steps, like the one-dimensional case, with a fixed multiplicative
constant correcting for the departure from that ideal. When $R$
heavily outweighs $Q$, the controller applies very little correction,
several eigenvalues of $A_{cl}$ end up close to one another instead of
one standing clearly apart, and the true value of $\alpha$ drifts, so a
constant fitted elsewhere no longer applies. Table~\ref{tab:bound_regimes}
shows this drift directly: $\alpha$ ranges from $1.23$ to $3.73$
depending on the regime, which is exactly the variation a single
calibrated constant cannot capture. The bound of
Proposition~\ref{prop:two_sided_bound} is unaffected by this drift
because it was never a point estimate: $\kappa_\infty(V)$ is
recomputed exactly for every $(Q,R)$ from the same eigendecomposition
that gives $\rho$, so the upper branch moves together with $\alpha$
instead of being fixed in advance. In practice this drift is not a
serious limitation for the lower branch, since $Q/R<0.1$ corresponds
to a controller that barely corrects the trajectory at all, a regime
with little engineering interest.
\end{remark}

Given that a single evaluation of \eqref{eq:two_sided_bound} requires
only the eigendecomposition of the $n\times n$ matrix $A_{cl}$ already
computed to obtain $\rho$, against a full closed-loop simulation for
the exact value, the bound gives, at essentially the same computational
cost as a calibrated point estimate, a guarantee the point estimate
could not offer: both branches hold for every $(Q,R)$, not only for
the ratios on which a constant happened to be calibrated. In fact, the
same closed-loop error recursion that gives the two-sided bound also
admits a direct, exact solution, without any need for an inequality at
all.

\vspace{3mm}

\begin{remark}[Towards an exact margin]
\label{rem:exact_margin}
Unrolling the error recursion \eqref{eq:err_dyn_recall} from $e_0=0$
shows that $\bar z_j$ in fact has a closed form,
\begin{equation}
\bar z_j = \sum_{i\ge0}\sum_l \bigl|[A_{cl}^{\,i}]_{jl}\bigr|\,\bar w_l,
\end{equation}
a convergent series rather than an inequality, whose supremum is
attained by a specific adversarial disturbance sequence. A Monte-Carlo
sequence reproduces this exact sign pattern with probability zero, so
the simulation-based estimate of $\bar z$ used earlier in this section
is, almost surely, a strict underestimate rather than a merely
approximate one. With a certified bound on the truncation of the
series, this exact value can be evaluated more cheaply than the DARE
solve it depends on, which removes the practical motivation for using
an approximate formula at all. The tube construction also imposes a
second tightening, on the input rather than the corridor, not
discussed here: the ancillary correction $-Ke_k$ must itself be
reserved out of $u_{\max}$, and this reservation can be severe in
practice, growing sharply with the sampling period. Both points are
left for future work.
\end{remark}

\section{Extension to an Arbitrary Reference Trajectory}
\label{sec:extension}

The sections above build the entire Tube MPC framework around a single
fixed reference: a circular orbit, linearised once and for all into
the CWH model. This section verifies that the same approach extends to
an arbitrary reference trajectory, one that is no longer a fixed orbit
but evolves in an unconstrained way over time, through a linearisation
recomputed at every instant along the path.

\subsection{Generating a Reference Trajectory Without a Closed Form}
\label{sec:ref_trajectory_generation}

To obtain a reference trajectory with no simple analytical expression,
we simulate a point mass subject to Earth's gravity and to a
constant-magnitude thrust whose direction rotates over time:
\[
  \ddot{\bm r}(t) = -\frac{\mu}{\Vert\bm r(t)\Vert^3}\bm r(t)
  + a_0\begin{pmatrix}\cos(\omega t)\\ \sin(\omega t)\cos\iota\\ \sin(\omega t)\sin\iota\end{pmatrix},
  \qquad
  \bm r(t) = \begin{pmatrix}x(t)\\y(t)\\z(t)\end{pmatrix} \in \R^3,
\]
where $\mu = 3.986\times10^{14}$\,m$^3$/s$^2$ is Earth's gravitational
parameter, $a_0 = 0.5g$ is the constant thrust magnitude, $\omega =
0.1$\textdegree/s is the angular rate at which the thrust direction
rotates, and $\iota = 25$\textdegree{} is an inclination angle ensuring
the thrust sweeps out of the initial orbital plane, producing a
genuinely three-dimensional trajectory.

This system is nonlinear (through the $\Vert\bm r\Vert^{-3}$ term) and
non-autonomous (the thrust depends explicitly on $t$), and therefore
admits no closed-form solution. It is integrated numerically, starting
from a circular orbit at 500\,km altitude, producing a sequence of
reference positions $\bm r_0^{\rm ref}, \bm r_1^{\rm ref}, \ldots,
\bm r_k^{\rm ref}, \ldots$, tracing a spiral trajectory with no
closed-form expression.

\subsection{Linearising the Dynamics Along the Trajectory}
\label{sec:linearising_along_trajectory}

Recall the dynamics under gravity alone,
$\bm a_{\rm grav}(\bm r) = -\mu\bm r/\Vert\bm r\Vert^3$, nonlinear in
$\bm r$. For CWH, this nonlinearity was handled once and for all by
linearising around a fixed point of the circular reference orbit.
Here, the reference $\bm r_k^{\rm ref}$ changes at every step $k$, so
this linearisation must be recomputed at every instant, around the
current reference point.

We seek an affine approximation of $\bm a_{\rm grav}$ in the vicinity
of a point $\bm r_0$:
\[
  \bm a_{\rm grav}(\bm r_0+\delta\bm r) \approx
  \bm a_{\rm grav}(\bm r_0) + J_g(\bm r_0)\,\delta\bm r,
\]
where $J_g(\bm r_0)\in\R^{3\times3}$ is the Jacobian of
$\bm a_{\rm grav}$, the matrix of partial derivatives
$(J_g)_{ij} = \partial a_{{\rm grav},i}/\partial r_j$. Writing out the
first component, $a_{{\rm grav},x} = -\mu x/\Vert\bm r\Vert^3$, by the
quotient rule:
\[
  \frac{\partial a_{{\rm grav},x}}{\partial x}
  = -\frac{\mu}{\Vert\bm r\Vert^3} + \frac{3\mu x^2}{\Vert\bm r\Vert^5},
  \qquad
  \frac{\partial a_{{\rm grav},x}}{\partial y}
  = \frac{3\mu\,xy}{\Vert\bm r\Vert^5},
  \qquad
  \frac{\partial a_{{\rm grav},x}}{\partial z}
  = \frac{3\mu\,xz}{\Vert\bm r\Vert^5}.
\]
The same computation for the remaining components, by the symmetric
roles of $x$, $y$, $z$ in $\Vert\bm r\Vert$, gives
\[
  \frac{\partial a_{{\rm grav},y}}{\partial y}
  = -\frac{\mu}{\Vert\bm r\Vert^3} + \frac{3\mu y^2}{\Vert\bm r\Vert^5},
  \qquad
  \frac{\partial a_{{\rm grav},z}}{\partial z}
  = -\frac{\mu}{\Vert\bm r\Vert^3} + \frac{3\mu z^2}{\Vert\bm r\Vert^5},
\]
and, for every off-diagonal pair $i\neq j$ with $i,j\in\{x,y,z\}$,
\[
  \frac{\partial a_{{\rm grav},i}}{\partial r_j}
  = \frac{3\mu\,r_ir_j}{\Vert\bm r\Vert^5}.
\]
Collecting all nine partial derivatives into a single matrix, and
writing $\bm r=(x,y,z)$ explicitly, gives the gravity Jacobian in full:
\begin{equation}
  \label{eq:gravity_jacobian}
  J_g(\bm r) =
  \begin{pmatrix}
    -\dfrac{\mu}{\Vert\bm r\Vert^3}+\dfrac{3\mu x^2}{\Vert\bm r\Vert^5}
      & \dfrac{3\mu xy}{\Vert\bm r\Vert^5}
      & \dfrac{3\mu xz}{\Vert\bm r\Vert^5} \\[3mm]
    \dfrac{3\mu xy}{\Vert\bm r\Vert^5}
      & -\dfrac{\mu}{\Vert\bm r\Vert^3}+\dfrac{3\mu y^2}{\Vert\bm r\Vert^5}
      & \dfrac{3\mu yz}{\Vert\bm r\Vert^5} \\[3mm]
    \dfrac{3\mu xz}{\Vert\bm r\Vert^5}
      & \dfrac{3\mu yz}{\Vert\bm r\Vert^5}
      & -\dfrac{\mu}{\Vert\bm r\Vert^3}+\dfrac{3\mu z^2}{\Vert\bm r\Vert^5}
  \end{pmatrix}
  = -\frac{\mu}{\Vert\bm r\Vert^3}I_3
  + \frac{3\mu}{\Vert\bm r\Vert^5}\,\bm r\,\bm r^\top,
\end{equation}
where $I_3$ is the $3\times3$ identity and $\bm r\bm r^\top$ is the
outer product of $\bm r$ with itself, with entries
$(\bm r\bm r^\top)_{ij} = r_ir_j$. This matrix is symmetric, a direct
consequence of the symmetry of $\bm r\bm r^\top$. The first term is
isotropic, reflecting an identical variation of gravity in every
direction; the second, proportional to $\bm r\bm r^\top$, reinforces
specifically the sensitivity along the radial direction. In practice,
$J_g$ must be re-evaluated at every step $k$, at the current reference
point $\bm r = \bm r_k^{\rm ref}$.

\subsubsection*{Assembling $A_c(t_k)$ and $B_c$}

Let $\bm\eta = (\bm r,\bm v)^\top \in \R^6$ denote the full
position-velocity state. Its evolution is governed by two relations:
the kinematic identity $\dot{\bm r} = \bm v$, and the linearised
dynamics $\dot{\bm v} \approx J_g(\bm r_k^{\rm ref})\,\delta\bm r +
\bm u$, where $\bm u$ is the thrust commanded by the controller.
Writing $\dot{\bm\eta} = A_c(t_k)\bm\eta + B_c\bm u$ and identifying
each $3\times3$ block of $A_c(t_k)$ with the corresponding partial
dependence:
\begin{itemize}
\item the $(\bm r,\bm r)$ block is $0_3$, since $\dot{\bm r}=\bm v$
      does not depend on $\bm r$;
\item the $(\bm r,\bm v)$ block is $I_3$, since $\dot{\bm r}=\bm v$
      depends on $\bm v$ with unit coefficient on each component;
\item the $(\bm v,\bm r)$ block is $J_g(\bm r_k^{\rm ref})$, the
      gravity Jacobian derived in \eqref{eq:gravity_jacobian};
\item the $(\bm v,\bm v)$ block is $0_3$, since the linearised gravity
      does not depend on velocity;
\end{itemize}
gives, written out in full and by direct analogy with the
construction of $A_c$ for CWH,
\begin{equation}
  \label{eq:Ac_time_varying}
  A_c(t_k) =
  \left(
  \begin{array}{ccc|ccc}
    0 & 0 & 0 & 1 & 0 & 0 \\
    0 & 0 & 0 & 0 & 1 & 0 \\
    0 & 0 & 0 & 0 & 0 & 1 \\ \hline
    (J_g)_{11} & (J_g)_{12} & (J_g)_{13} & 0 & 0 & 0 \\
    (J_g)_{21} & (J_g)_{22} & (J_g)_{23} & 0 & 0 & 0 \\
    (J_g)_{31} & (J_g)_{32} & (J_g)_{33} & 0 & 0 & 0
  \end{array}
  \right) \in \R^{6\times6},
\end{equation}
where each $(J_g)_{ij}$ is one of the nine explicit entries of
\eqref{eq:gravity_jacobian}, evaluated at $\bm r = \bm r_k^{\rm ref}$.
The same reasoning applied to the input matrix $B_c$ gives its two
blocks: the commanded thrust $\bm u$ has no direct effect on position
($(\bm r,\bm u)$ block equal to $0_3$), but acts directly as an
acceleration, one component at a time ($(\bm v,\bm u)$ block equal to
$I_3$), so that
\begin{equation}
  \label{eq:Bc_definition}
  B_c =
  \left(
  \begin{array}{ccc}
    0 & 0 & 0 \\
    0 & 0 & 0 \\
    0 & 0 & 0 \\ \hline
    1 & 0 & 0 \\
    0 & 1 & 0 \\
    0 & 0 & 1
  \end{array}
  \right) \in \R^{6\times3}.
\end{equation}
Unlike CWH, where $A_c$ depends only on the constant orbital rate $n$,
every entry of $A_c(t_k)$ inherited from $J_g$ depends explicitly on
the current reference position $\bm r_k^{\rm ref}$, so this produces a
sequence of distinct matrices $A_c(t_0), A_c(t_1), \ldots$, one per
step of the simulated trajectory, while $B_c$ itself remains constant,
since it does not depend on position.

\subsubsection*{Discretisation}
\vspace{-1mm}
Each pair $(A_c(t_k), B_c)$ is discretised by the exact matrix
exponential, as for CWH, but recomputed at every instant:
\begin{equation}
  \label{eq:Ad_time_varying}
  A_d(t_k) = e^{A_c(t_k)\,T_s},
  \qquad
  B_d(t_k) = \int_0^{T_s} e^{A_c(t_k)\tau}\,B_c\,\mathrm d\tau,
\end{equation}
both obtained in a single matrix exponentiation of the augmented block
$\begin{pmatrix}A_c(t_k) & B_c\\ 0 & 0\end{pmatrix}$, where $T_s$ is
the controller's sampling period. Over the simulated trajectory this
produces 361 distinct pairs $(A_d(t_k), B_d(t_k))$, one per step, in
place of the single fixed pair used for CWH.

\subsection{A Time-Varying Tube MPC}
\label{sec:time_varying_tube_mpc}

The Tube MPC gain, normally the unique solution of the DARE for a
fixed dynamics matrix,
\vspace{-2mm}
\[
  P = Q + A_d^\top PA_d - A_d^\top PB_d(R+B_d^\top PB_d)^{-1}B_d^\top PA_d,
  \qquad
  K = (R+B_d^\top PB_d)^{-1}B_d^\top PA_d,
\]
must instead be solved locally at every step $k$, using
$(A_d(t_k), B_d(t_k))$:
\begin{equation}
  \label{eq:local_dare}
  P_k = Q + A_d(t_k)^\top P_k A_d(t_k)
  - A_d(t_k)^\top P_k B_d(t_k)\bigl(R+B_d(t_k)^\top P_k B_d(t_k)\bigr)^{-1}
    B_d(t_k)^\top P_k A_d(t_k),
\end{equation}
\begin{equation}
  \label{eq:local_gain}
  K(t_k) = \bigl(R+B_d(t_k)^\top P_k B_d(t_k)\bigr)^{-1}B_d(t_k)^\top P_k A_d(t_k),
\end{equation}
producing a gain specific to every instant, solved 361 times rather
than once. The applied control retains exactly the form of the
time-invariant Tube MPC:
\begin{equation}
  \label{eq:control_time_varying}
  u_k = v_k - K(t_k)\,e_k, \qquad e_k = x_k - z_k,
\end{equation}
where $v_k$ is the nominal control obtained from the condensed QP
built locally from $(A_d(t_k), B_d(t_k))$ over a receding horizon of
$N$ steps.
\vspace{-1mm}
\subsection{Numerical Validation}
\label{sec:extension_validation}

We reuse the robustness test parameters already employed in Sections~5
and~6: horizon $N=100$, sampling period $T_s=10$\,s, safety corridor
$\varepsilon=100$\,m, and disturbance bounds
$\bar w_{\rm pos}=10$\,m, $\bar w_{\rm vel}=1$\,m/s. Thirty Monte Carlo
simulations, each with an independent disturbance realisation
$\xi_k\sim\mathcal U([-\bar w,\bar w])$, verify that
$\Vert x_k - x_k^{\rm ref}\Vert \leq \varepsilon$ at every step.

\begin{figure}[H]
\centering
\includegraphics[width=0.6\textwidth]{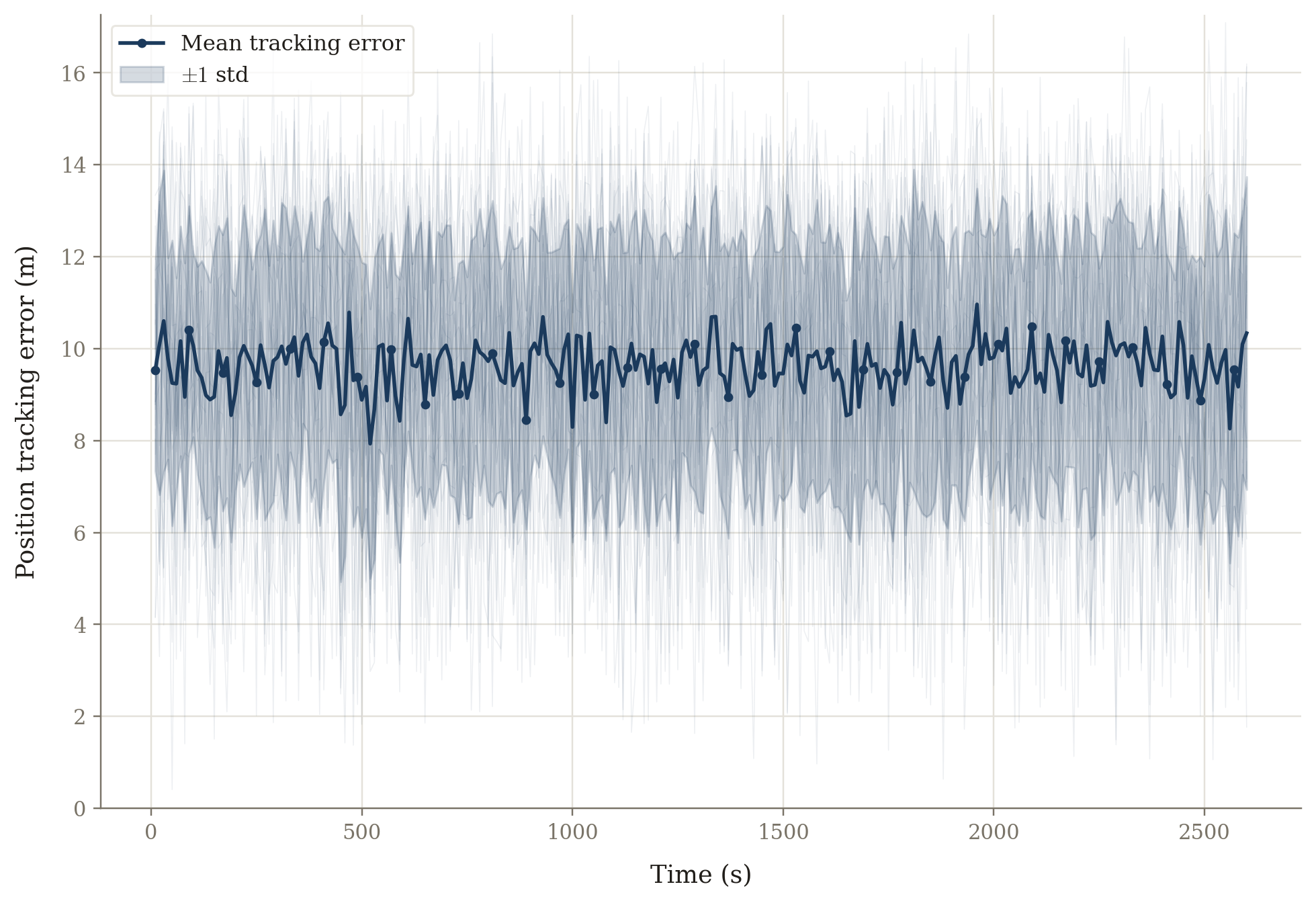}
\caption{Position tracking error over time, 30 Monte Carlo runs,
reference trajectory generated by gravity and rotating thrust
($N=100$, $T_s=10$\,s, $\bar w_{\rm pos}=10$\,m, $\bar w_{\rm
vel}=1$\,m/s). Thin lines: individual runs. Bold line: mean over 30
runs. Shaded band: one standard deviation. The chaser starts with an
initial offset of about 25.5\,m from the reference and converges
towards a steady-state tracking error, well inside the 100\,m safety
corridor throughout.}
\label{fig:monte_carlo_corridor}
\end{figure}

\vspace{-4mm}

\begin{table}[H]
\centering
\caption{Tracking results over 30 Monte Carlo runs, reference
trajectory generated by gravity and rotating thrust, $N=100$,
$T_s=10$\,s, $\varepsilon=100$\,m.}
\renewcommand{\arraystretch}{1.3}
\vspace{6mm}
\begin{tabular}{lc}
\hline
\textbf{Metric} & \textbf{Value} \\
\hline
Final error (mean $\pm$ std) & $10.34 \pm 3.40$\,m \\
Maximum error (mean $\pm$ std) & $16.04 \pm 0.58$\,m \\
Maximum error observed & $17.09$\,m \\
Corridor violations (out of 30) & $0$ (0\%) \\
\hline
\end{tabular}
\end{table}

Across all 30 realisations, the chaser never leaves the safety
corridor, with a comfortable margin: the largest observed error
(17.09\,m) represents only 17\% of the allowed corridor. This confirms
that the local linearisation approach, recomputed at every instant
along an arbitrary reference trajectory, preserves the same robustness
guarantee established for the standard CWH manoeuvres, without
requiring any further reformulation of the controller.

\section{Towards a Flight-Ready Application}
\label{sec:embedded_no_libs}

All numerical results presented so far rely on \texttt{numpy} and
\texttt{scipy} for the underlying linear algebra and numerical
integration, a natural choice for the exploratory work carried out
throughout this report. It does not, however, reflect the constraints
under which real flight software is developed. In an embedded avionics
context, as is standard practice at spacecraft manufacturers such as
ArianeGroup, the onboard flight computer runs certified,
resource-constrained software, and a general-purpose scientific library
is almost never available on the target hardware: every line of code
running on board must be reviewed and certified, and a library such as
\texttt{scipy} is far too large and general to be certified as a whole.
Every numerical routine actually used by the guidance software must
therefore be written from scratch, typically in C or Ada, using only
basic arithmetic and explicit loops.

To assess how much of the Tube MPC + ADMM pipeline would need to be
re-implemented under such constraints, every routine relying on
\texttt{numpy}, \texttt{scipy}, or \texttt{math} was rewritten from
first principles. None of the underlying numerical methods are new. In
particular, the linear solver used throughout this section rests on the
following classical result.

\vspace{3mm}
\begin{theorem}[Cholesky, \cite{Golub1996}]
\label{thm:cholesky}
Let $A\in\R^{n\times n}$ be symmetric ($A^\top=A$) and positive
definite ($\bm x^\top A\bm x>0$ for every nonzero $\bm x\in\R^n$).
Then there exists a unique lower-triangular matrix $L\in\R^{n\times n}$
with strictly positive diagonal entries such that $A=LL^\top$.
\end{theorem}
\vspace{2mm}
Both hypotheses hold for every matrix factorised in this report:
symmetry follows from the definitions of $H$, $R$ and $P$
(Sections~\ref{sec:linear_mpc} and~\ref{sec:tube_mpc_recall}), and
positive definiteness from $R\succ0$, as established for the matrix $M$
inside Algorithm~\ref{alg:riccati} below. The Riccati iteration and
Runge-Kutta integration used below are likewise textbook material, left
unchanged from their standard formulation: the purpose here is one of
implementation, not algorithmic contribution, demonstrating that the
pipeline can be made fully self-contained with traceable, line-by-line
behaviour of the kind expected during flight software qualification.

Not every rewritten routine required a genuine algorithm. Basic array
operations, such as constructing a zero or identity matrix, transposing
a matrix, or computing a matrix-vector product, are direct, mechanical
applications of their definition and are omitted below for brevity.
Three routines, by contrast, compute a quantity with no closed-form
expression and require a genuine iterative process or factorisation;
these are presented in detail in the remainder of this section.

\subsection{Algorithm 1: Solving a Symmetric Positive-Definite Linear System}
At every ADMM iteration \eqref{eq:admm_v}, the $\bm v$-update requires
solving a linear system $(H+\rho I)\bm v = \bm b$, where $H+\rho I$ is
symmetric positive definite. This step replaces
\texttt{np.linalg.solve} and \texttt{scipy.linalg.cho\_factor} /
\texttt{cho\_solve}. In the notation of Algorithm~\ref{alg:cholesky}
below, $A = H+\rho I$ is the matrix to factorise, $\bm b$ is the
right-hand side $-f+\rho(\bm z^{(m)}-\bm d^{(m)})$, and $\bm x$ is the
unknown $\bm v^{(m+1)}$ solved for at ADMM iteration $m$. Rather than
inverting the matrix explicitly, the system is solved through a
Cholesky factorisation $A=LL^\top$, originally devised by
André-Louis Cholesky for the solution of normal equations in
geodetic triangulation and published posthumously by his colleague
Commandant Benoît~\cite{Benoit1924}, followed by two triangular solves.

\begin{algorithm}[H]
\caption{Cholesky factorisation and triangular solve}
\label{alg:cholesky}
\begin{algorithmic}[1]
\Require Symmetric positive-definite matrix $A \in \R^{n\times n}$, right-hand side $\bm b \in \R^n$
\Ensure Solution $\bm x$ of $A\bm x = \bm b$
\State Compute the Cholesky factor $L$ such that $A = LL^\top$:
\For{$i = 1$ to $n$}
  \For{$j = 1$ to $i$}
    \State $s \leftarrow \sum_{k=1}^{j-1} L_{ik}L_{jk}$
    \If{$i = j$}
      \State $L_{ii} \leftarrow \sqrt{A_{ii} - s}$
    \Else
      \State $L_{ij} \leftarrow \dfrac{1}{L_{jj}}\bigl(A_{ij} - s\bigr)$
    \EndIf
  \EndFor
\EndFor
\State Solve $L\bm y = \bm b$ by forward substitution:
\For{$i = 1$ to $n$}
  \State $y_i \leftarrow \dfrac{1}{L_{ii}}\Bigl(b_i - \sum_{k=1}^{i-1}L_{ik}y_k\Bigr)$
\EndFor
\State Solve $L^\top\bm x = \bm y$ by backward substitution:
\For{$i = n$ downto $1$}
  \State $x_i \leftarrow \dfrac{1}{L_{ii}}\Bigl(y_i - \sum_{k=i+1}^{n}L_{ki}x_k\Bigr)$
\EndFor
\State \Return $\bm x$
\end{algorithmic}
\end{algorithm}

To illustrate Algorithm~\ref{alg:cholesky} on a readable example, consider a
symmetric positive-definite matrix $A\in\R^{6\times6}$ (the same
dimension as the CWH state) and a right-hand side $\bm b$:
\[
A =
\begin{pmatrix}
9.87 & 1.59 & -0.06 & 1.60 & -1.18 & 0.04 \\
1.59 & 10.90 & -1.03 & 4.35 & 0.17 & -1.84 \\
-0.06 & -1.03 & 8.82 & -0.05 & 4.31 & -2.26 \\
1.60 & 4.35 & -0.05 & 15.42 & -0.07 & -5.03 \\
-1.18 & 0.17 & 4.31 & -0.07 & 14.78 & -3.97 \\
0.04 & -1.84 & -2.26 & -5.03 & -3.97 & 15.42
\end{pmatrix},
\qquad
\bm b =
\begin{pmatrix}
-1.204 \\ 1.462 \\ 1.766 \\ -0.329 \\ 0.841 \\ -0.180
\end{pmatrix}.
\]
Step~1 of Algorithm~\ref{alg:cholesky} produces the lower-triangular
factor
\[
L =
\begin{pmatrix}
3.141 & 0 & 0 & 0 & 0 & 0 \\
0.505 & 3.262 & 0 & 0 & 0 & 0 \\
-0.020 & -0.313 & 2.953 & 0 & 0 & 0 \\
0.510 & 1.253 & 0.119 & 3.685 & 0 & 0 \\
-0.374 & 0.111 & 1.469 & -0.053 & 3.531 & 0 \\
0.012 & -0.564 & -0.825 & -1.148 & -0.779 & 3.534
\end{pmatrix},
\]

which satisfies $\|LL^\top-A\|_\infty=1.8\times10^{-15}$, confirming
the factorisation is correct up to floating-point round-off. Steps~2
and~3 then give the solution vector

\[
\bm x_{\rm manual} =
\begin{pmatrix}
-0.147141 \\ 0.204653 \\ 0.239313 \\ -0.054929 \\ -0.020601 \\ 0.024896
\end{pmatrix},
\qquad
\bm x_{\rm exact} \;(\texttt{np.linalg.solve}) =
\begin{pmatrix}
-0.147141 \\ 0.204653 \\ 0.239313 \\ -0.054929 \\ -0.020601 \\ 0.024896
\end{pmatrix},
\]
with $\|\bm x_{\rm manual}-\bm x_{\rm exact}\|_\infty=5.6\times10^{-17}$,
that is, exact agreement up to machine precision. On the larger
condensed system actually solved at every ADMM iteration ($N=20$,
$n=60$ unknowns), the same agreement holds, at
$1.4\times10^{-17}$. Executing the factorisation and both substitutions
in pure Python loops takes 19.9\,ms on this $n=60$ system, against
0.5\,ms for \texttt{np.linalg.solve}, a factor of roughly 39 that
reflects the overhead of interpreted Python loops rather than any
property of the algorithm itself; the same code compiled in C, as it
would be onboard, closes this gap almost entirely. Algorithm~1 is exact
and requires no tolerance: it terminates after a fixed, predictable
number of operations ($O(n^3)$ for the factorisation, $O(n^2)$ for each
substitution), a property valued in flight software, where worst-case
execution time must be bounded in advance.

\subsection{Algorithm 2: Solving the Discrete Algebraic Riccati Equation}

The routine \texttt{scipy.linalg.solve\_discrete\_are} is replaced
here, since no closed-form expression gives $P$ directly from
\eqref{eq:DARE}. Both $Q$ and $R$ are symmetric by construction
(Section~\ref{sec:linear_mpc}), and so is $P$ at every stage of the
iteration described below, a property that follows directly from the
symmetry of $Q$ and is preserved by the Riccati map at each step.
Starting from $P^{(0)}=Q$, the Riccati map is applied repeatedly until
convergence, with each iteration relying on
Algorithm~\ref{alg:cholesky} to evaluate the matrix inverse term,
which itself requires a symmetric positive-definite matrix as input.

\begin{algorithm}[H]
\caption{Riccati fixed-point iteration for the DARE}
\label{alg:riccati}
\begin{algorithmic}[1]
\Require System matrices $A_d, B_d$, weight matrices $Q \succeq 0$, $R \succ 0$, tolerance $\varepsilon$, maximum iterations $I_{\max}$
\Ensure Solution $P$ of the DARE \eqref{eq:DARE}
\State $P \leftarrow Q$
\For{$i = 1$ to $I_{\max}$}
  \State $M \leftarrow R + B_d^\top P B_d$
  \State Solve $M\, K_{\rm term} = B_d^\top P A_d$ for $K_{\rm term}$ using Algorithm~\ref{alg:cholesky}
  \State $P_{\rm new} \leftarrow Q + A_d^\top P A_d - A_d^\top P B_d\, K_{\rm term}$
  \If{$\|P_{\rm new} - P\|_\infty < \varepsilon$}
    \State \Return $P_{\rm new}$
  \EndIf
  \State $P \leftarrow P_{\rm new}$
\EndFor
\State \Return $P$
\end{algorithmic}
\end{algorithm}

Before applying Algorithm~\ref{alg:cholesky} inside the Riccati
iteration, it is worth verifying that the matrix $M=R+B_d^\top PB_d$ 
solved at each step does satisfy the required hypothesis, namely that
$M$ is symmetric positive definite. Symmetry follows directly from
that of $R$ and $P$: taking the transpose of $M$ and using
$(B_d^\top PB_d)^\top = B_d^\top P^\top B_d$,
\[
M^\top = R^\top + B_d^\top P^\top B_d = R + B_d^\top P B_d = M,
\]
since $R^\top=R$ and $P^\top=P$ by construction. Positive definiteness
follows from $R\succ0$ together with $B_d^\top PB_d\succeq0$ whenever
$P\succeq0$, since for any nonzero $\bm x$,
\[
\bm x^\top M\bm x = \underbrace{\bm x^\top R\bm x}_{>0}
+ \underbrace{\bm x^\top B_d^\top PB_d\bm x}_{\ge0} > 0.
\]
Both properties are therefore guaranteed at every iteration, and
Algorithm~\ref{alg:cholesky} can be applied to $M$ without further
justification.

To illustrate Algorithm~\ref{alg:riccati} concretely, consider a
six-dimensional CWH system, with state weight
$Q=3\times10^2\,\mathrm{diag}(100,100,100,1,1,1)$ and control weight
$R=100I_3$. Starting from $P^{(0)}=Q$ and iterating according to
Algorithm~\ref{alg:riccati} with tolerance $\varepsilon=10^{-10}$, the
sequence $P^{(i)}$ converges after 366 steps to
\[
P_{\rm manual} =
\begin{pmatrix}
30302.0 & 0 & 0 & 1510.0 & 5.6 & 0 \\
0 & 30302.0 & 0 & -5.6 & 1510.0 & 0 \\
0 & 0 & 30302.0 & 0 & 0 & 1510.0 \\
1510.0 & -5.6 & 0 & 7851.7 & 0 & 0 \\
5.6 & 1510.0 & 0 & 0 & 7851.9 & 0 \\
0 & 0 & 1510.0 & 0 & 0 & 7851.9
\end{pmatrix},
\]
to be compared with the reference value obtained from
\texttt{scipy.linalg.solve\_discrete\_are},
\[
P_{\rm scipy} =
\begin{pmatrix}
30302.0 & -0.001 & 0 & 1510.0 & 5.6 & 0 \\
-0.001 & 30302.0 & 0 & -5.6 & 1510.0 & 0 \\
0 & 0 & 30302.0 & 0 & 0 & 1510.0 \\
1510.0 & -5.6 & 0 & 7851.7 & -0.021 & 0 \\
5.6 & 1510.0 & 0 & -0.021 & 7851.9 & 0 \\
0 & 0 & 1510.0 & 0 & 0 & 7851.9
\end{pmatrix},
\]
which agree to within $\|P_{\rm manual}-P_{\rm scipy}\|_\infty=3.6\times10^{-6}$,
already six orders of magnitude below the entries of $P$. Note that
this final gain $K$ is distinct from the intermediate $K_{\rm term}$
solved for at every iteration inside Algorithm~\ref{alg:riccati}: the
latter is discarded once $P_{\rm new}$ has been formed and is never
reused, whereas $K$ is computed only once, after convergence, by
applying
\[
K = (R + B_d^\top P B_d)^{-1} B_d^\top P A_d
\]
to the final value $P_{\rm manual}$ rather than to an intermediate
iterate. The gain $K$ derived this way from $P$, which is the quantity
actually used in the control law $u_k=v_k-Ke_k$, agrees even more
closely: comparing
\[
K_{\rm manual} =
\begin{pmatrix}
0.0196 & -0.0001 & 0 & 0.198 & 0.0007 & 0 \\
0.0001 & 0.0196 & 0 & -0.0007 & 0.198 & 0 \\
0 & 0 & 0.0196 & 0 & 0 & 0.198
\end{pmatrix}
\]

with

\[
K_{\rm scipy} =
\begin{pmatrix}
0.0196 & -0.0001 & 0 & 0.198 & 0.0007 & 0 \\
0.0001 & 0.0196 & 0 & -0.0007 & 0.198 & 0 \\
0 & 0 & 0.0196 & 0 & 0 & 0.198
\end{pmatrix}
\]

\vspace{4mm}
gives $\|K_{\rm manual}-K_{\rm scipy}\|_\infty=8.8\times10^{-13}$, that
is, agreement to thirteen significant digits. Table~\ref{tab:dare_tol}
reports how the error on $P$ and the required number of iterations
depend on the chosen tolerance $\varepsilon$.

\vspace{3mm}

\begin{table}[H]
\centering
\caption{Convergence of Algorithm~\ref{alg:riccati} for different
tolerances $\varepsilon$, on the six-dimensional CWH system.}
\vspace{3mm}
\renewcommand{\arraystretch}{1.3}
\begin{tabular}{|c|c|c|l|}
\hline
\textbf{Tolerance $\varepsilon$} & \textbf{Iterations} & \textbf{Error on $P$} & \textbf{Comment} \\
\hline
\rowcolor{green!15}
$10^{-6}$  & 261 & $1.5\times10^{-5}$ & \textbf{recommended}: best accuracy-to-cost ratio \\
\hline
$10^{-8}$  & 318 & $3.7\times10^{-6}$ & modest accuracy gain, 22\% more iterations \\
\hline
$10^{-10}$ & 366 & $3.6\times10^{-6}$ & marginal accuracy gain, 40\% more iterations \\
\hline
$10^{-12}$ & $>2000$ & $3.6\times10^{-6}$ & stagnates on floating-point round-off \\
\hline
\end{tabular}
\label{tab:dare_tol}
\end{table}
Tightening $\varepsilon$ below $10^{-6}$ brings only marginal improvement
in accuracy: the error on $P$ moves from $1.5\times10^{-5}$ to
$3.7\times10^{-6}$ between $\varepsilon=10^{-6}$ and $\varepsilon=10^{-8}$,
a modest gain paid for with 57 additional iterations, and stagnates
entirely below $\varepsilon=10^{-10}$. Pushing further to
$\varepsilon=10^{-12}$ fails to converge within 2000 iterations, for no
gain in accuracy on $P$ at all. The recommended setting is therefore
$\varepsilon=10^{-6}$ with $I_{\max}=500$, which converges in 261
iterations. Since $P$ is only ever computed offline in the
time-invariant case (Sections~5--7), or once per guidance step in the
time-varying case of Section~\ref{sec:extension}, this cost remains
acceptable within the sampling period $T_s$. Algorithm~2 requires
$O(I\cdot n_x^3)$ operations, dominated by the three $n_x\times n_x$
matrix products performed at every iteration, with $I$ the number of
iterations to convergence; this cost is independent of the horizon $N$,
since neither $A_d$, $B_d$, $Q$ nor $R$ scale with $N$.

\subsection{Complexity and Memory of the Tube MPC / ADMM Pipeline}
\label{sec:complexity}

Before presenting the complexity and memory tables, it is worth fixing
precisely what each symbol below stands for, since several of them
have already appeared with a slightly different scope earlier in this
report.

\begin{itemize}
  \item $n_x$ : dimension of the state vector ($n_x=4$ for the planar
        rendezvous of Section~\ref{sec:num_results}, $n_x=6$ for the
        three-dimensional manoeuvres of Section~\ref{sec:tube_mpc_recall}
        and the arbitrary trajectory of Section~\ref{sec:extension});
  \item $n_u$ : dimension of the control input ($n_u=2$ or $n_u=3$
        respectively);
  \item $N$ : prediction horizon, the number of future steps
        considered at every MPC call;
  \item $M$ : number of ADMM iterations performed at every MPC step
        to solve the condensed QP online;
  \item $I$ : number of Riccati fixed-point iterations
        (Algorithm~\ref{alg:riccati}) required for the DARE solve to
        converge.
\end{itemize}

The role played by $I$, and more generally the split between what is
computed once and what is repeated at every guidance step, is not the
same for the two settings studied in this report:

\begin{itemize}
  \item In the \textbf{time-invariant CWH case}
        (Sections~\ref{sec:num_results} to \ref{sec:simplification}),
        the dynamics matrices $A_d,B_d$ are fixed for the entire
        mission, so the condensed matrices $\mathcal A,\mathcal B,H$,
        the Cholesky factor $L$, and the DARE solution $P,K$ are all
        computed exactly once, before launch; only the ADMM loop and
        the control law genuinely repeat online, at every sampling
        period $T_s$.
  \item In the \textbf{time-varying case} of
        Section~\ref{sec:extension}, the reference trajectory
        $r_k^{\rm ref}$ moves at every step, so $A_d(t_k),B_d(t_k)$
        change as well, and with them $\mathcal A,\mathcal B,H,L,P,K$:
        every quantity that was computed once offline in the CWH case
        must instead be recomputed at every one of the $361$ guidance
        steps used in this section's validation.
\end{itemize}

Table~\ref{tab:cwh_complexity} and Table~\ref{tab:nonlinear_complexity}
report the two regimes separately for this reason.

\subsubsection*{Case 1: Time-invariant CWH dynamics}

\begin{table}[htbp]
\centering
\renewcommand{\arraystretch}{1.4}
\begin{tabular}{|p{7.5cm}|c|c|}
\hline
\textbf{Step} & \textbf{Time} & \textbf{Memory} \\
\hline
\rowcolor{gray!12}
\multicolumn{3}{|l|}{\textit{Computed once, offline, before the mission}} \\
\hline
$A_d,B_d$ (degree-3 truncation) & $O(n_x^3)$ & $O(n_x^2)$ \\
\hline
Condensed $\mathcal A$ (stack of $A_d^m$) & $O(N n_x^3)$ & $O(N n_x^2)$ \\
\hline
Condensed $\mathcal B$ (block lower-triangular) & $O(N^2 n_x^2 n_u)$ & $O(N^2 n_x n_u)$ \\
\hline
Hessian $H=\mathcal B^\top\bar Q\mathcal B+\bar R$ & $O(n_u^2 n_x N^3)$ & $O(n_u^2 N^2)$ \\
\hline
Cholesky factorisation of $H+\rho I$ (Algorithm~1) & $O(n_u^3 N^3)$ & $O(n_u^2 N^2)$ \\
\hline
DARE solve, $I$ Riccati iterations (Algorithm~2) & $O(I\,n_x^3)$ & $O(n_x^2)$ \\
\hline
\rowcolor{gray!12}
\multicolumn{3}{|l|}{\textit{Repeated online, at every guidance step $k$}} \\
\hline
Per-step vector $\bm f$ & $O(n_u n_x N^2)$ & $O(n_u N)$ \\
\hline
ADMM loop, $M$ iterations (Cholesky factor $L$ reused) & $O(M\,n_u^2N^2)$ & $O(n_u N)$ \\
\hline
Control law $\bm u_k=\bm v_0-K\bm e_k$, state update & $O(n_x^2)$ & $O(n_x)$ \\
\hline
\rowcolor{green!25!black!10}
\textbf{Total per MPC step (online)} & \textbf{$O\big(N^2(n_un_x+Mn_u^2)\big)$} & \textbf{$O(n_u^2N^2)$} \\
\hline
\end{tabular}
\caption{Time-invariant CWH case: time complexity and memory footprint
of every step of the dependency-free Tube MPC + ADMM pipeline. All six
offline quantities are computed once, before the mission, and reused
unchanged at every guidance step.}
\label{tab:cwh_complexity}
\end{table}

In this regime, the online cost is dominated by the two triangular
solves inside the ADMM $\bm v$-update, since the Cholesky factorisation
of $H+\rho I$, the single most expensive step in the entire offline
pipeline at $O(n_u^3N^3)$, is performed once and its factor $L$ is
simply reused at every ADMM iteration and at every MPC step: this is
precisely why the online cost, $O(M\,n_u^2N^2)$, is one full power of
$N$ cheaper than the $O(n_u^3N^3)$ factorisation it relies on. For the
largest configuration used in this report ($n_x=6$, $n_u=3$, $N=100$,
$M=100$), this gives an online cost per MPC step of $9.0\times10^6$
elementary operations, against $2.7\times10^7$ operations if the
factorisation were, incorrectly, repeated at every ADMM iteration
instead of being reused: a factor of $300$ in this configuration. The
online memory footprint is likewise dominated by $L$, which must
remain resident onboard between MPC steps: storing only its nonzero
lower-triangular entries, $\frac{n_uN(n_uN+1)}{2}$ scalars, requires
$352.7$\,KB at $n_x=6$, $n_u=3$, $N=100$ (double precision), against
$703.1$\,KB for a dense $n_uN\times n_uN$ matrix that ignores the
triangular structure, a further factor of two saved at no
implementation cost, on top of the block-wise application of $\bar Q$
already discussed in Section~\ref{sec:simplification}.

\subsubsection*{Case 2: Time-varying, nonlinear reference trajectory}

\begin{table}[htbp]
\centering
\renewcommand{\arraystretch}{1.4}
\begin{tabular}{|p{8.5cm}|c|c|}
\hline
\textbf{Step (repeated at every guidance step $k$)} & \textbf{Time} & \textbf{Memory} \\
\hline
Gravity Jacobian $J_g(r_k^{\rm ref})$ (nine explicit entries) & $O(1)$ & $O(1)$ \\
\hline
$A_d(t_k),B_d(t_k)$ (matrix exponential of augmented block) & $O(n_x^3)$ & $O(n_x^2)$ \\
\hline
Condensed $\mathcal A(t_k),\mathcal B(t_k)$ (rebuilt from scratch) & $O(N^2 n_x^2 n_u)$ & $O(N^2 n_x n_u)$ \\
\hline
Hessian $H(t_k)$ (rebuilt from scratch) & $O(n_u^2 n_x N^3)$ & $O(n_u^2 N^2)$ \\
\hline
Cholesky refactorisation of $H(t_k)+\rho I$ & $O(n_u^3 N^3)$ & $O(n_u^2 N^2)$ \\
\hline
DARE re-solve, $I$ Riccati iterations, for $P_k,K(t_k)$ & $O(I\,n_x^3)$ & $O(n_x^2)$ \\
\hline
ADMM loop, $M$ iterations, using the freshly refactored $L$ & $O(M\,n_u^2N^2)$ & $O(n_u N)$ \\
\hline
Control law $\bm u_k=\bm v_k-K(t_k)\bm e_k$, state update & $O(n_x^2)$ & $O(n_x)$ \\
\hline
\rowcolor{green!25!black!10}

\textbf{Total per guidance step} & \textbf{$O(n_u^2N^3(n_x+n_u))$} & \textbf{$O(n_u^2N^2)$} \\

\hline
\end{tabular}
\caption{Time-varying, nonlinear case: time complexity and memory
footprint of every step of the pipeline, all of which are now repeated
at every one of the $361$ guidance steps used in the validation of
this section, since the reference trajectory changes at every step and
none of the offline reuse of Table~\ref{tab:cwh_complexity} is
available.}
\label{tab:nonlinear_complexity}
\end{table}

Compared with the CWH case, the change is structural rather than a
difference in the asymptotic exponents themselves: the first six rows
of Table~\ref{tab:cwh_complexity}, computed once for the whole mission,
must now be recomputed at every one of the $361$ guidance steps, so the
online cost is dominated by the $O(n_u^2N^3(n_x+n_u))$ cost of rebuilding
the Hessian and refactorising it via Cholesky, both $O(N^3)$ and
therefore negligible only when amortised over an entire mission, as in
the CWH case. Numerically, this gives roughly $8.1\times10^7$ operations
per step (of which $5.4\times10^7$ from building $H$ and $2.7\times10^7$
from the Cholesky refactorisation) against $9.0\times10^6$ in the CWH
case: about nine times more expensive per step, consistent with the
observation, already made in Section~\ref{sec:complexity}, that none
of the offline reuse available in the time-invariant case survives
once the reference trajectory itself is time-varying, even though both
settings share the same asymptotic building blocks. The memory
footprint stays the same order, $O(n_u^2N^2)$ dominated by $L$, but is
now overwritten at every step rather than loaded once for the mission.

\newpage

\section{Conclusion}
\label{sec:conclusion}

This report set out to build and stress-test a family of real-time
MPC controllers for autonomous spacecraft rendezvous. Among the four
algorithmic families studied, Tube MPC consistently stood out: it
delivered the most accurate and robust tracking on the original
benchmark, handled all five standard rendezvous manoeuvres with a
single fixed configuration, and kept performing reliably even once the
horizon was pushed far beyond what the original design anticipated.
Lengthening the prediction horizon exposed a real weakness in
gradient-based solvers such as Fast MPC: the marginal stability of the
CWH dynamics makes the underlying optimisation problem increasingly
ill-conditioned as the horizon grows, and solvers relying on exact
linear algebra, ADMM in particular, are simply immune to this effect
while Nesterov's method is not. This distinction, first observed as a
numerical curiosity, explains a large part of the behaviour seen
throughout the report.

The remainder of the work builds on that understanding. A polynomial
approximation of the dynamics matrices made the long-horizon regime
computationally affordable without compromising accuracy, confirmed
across a thousand random test cases. A systematic look at the cost
weights showed that fuel consumption responds strongly to how they are
tuned while tracking accuracy barely moves, and a two-sided, provably
correct bound was derived to bracket the worst-case tracking error
directly from these weights, together with a first characterisation of
the exact closed form that the bound approximates. The framework was
then pushed past its original assumption of a fixed circular orbit: a
reference trajectory with no closed-form description was tracked
reliably once the underlying linearisation was recomputed along the
way, with the safety guarantee holding throughout. Taken together,
these results make a strong case for Tube MPC as the most dependable
option among those considered, combining a real robustness guarantee
with a computational cost realistic for onboard, real-time use.

Finally, the report asked whether this pipeline could actually run on
the hardware it is designed for. The two routines relying on
\texttt{numpy} or \texttt{scipy} inside the real-time loop were
re-implemented from first principles, using only elementary arithmetic
and explicit loops, demonstrating that the Tube MPC + ADMM pipeline
can be made fully self-contained, with traceable, line-by-line
numerical behaviour, at a modest computational cost. A complexity and
memory analysis of the full pipeline showed that the two settings
studied here behave very differently: in the time-invariant CWH case,
the most expensive steps are computed once offline and amortised over
the entire mission, leaving an online cost per guidance step that
scales as $O(N^2(n_un_x+Mn_u^2))$; in the time-varying case, every
one of these steps must be recomputed at each guidance step, pushing
the online cost to $O(n_u^2N^3(n_x+n_u))$, roughly nine times higher
in the configuration tested here. This gap is not a flaw but a direct
and quantified consequence of tracking a reference that changes at
every step, and it is precisely the kind of figure a flight software
budget needs before committing to either regime.

The extension to time-varying reference trajectories was tested on a
single scenario, and a broader sweep across orbital regimes would help
confirm how general the conclusion is. A natural next step would also
be to profile the dependency-free implementation directly on
embedded-class hardware, to confirm that the predicted operation counts
translate into the expected execution times once compiled. The most
natural continuation of this work, however, would be to consolidate
the dependency-free pipeline of Section~\ref{sec:embedded_no_libs}
into a portable implementation in C or Ada, and to deploy it on a real
embedded platform, so that the Tube MPC / ADMM controller can be
tested against an actual physical system rather than a software
simulator. This would close the loop between the theoretical guarantees
established here and the operational constraints of a real guidance
computer, and represents the clearest path from the present report
towards a flight-ready application.

All the code developed for this report is available at
\url{https://github.com/adnanegrb/MPC-spacecraft-guidance}: the MPC
families and solvers, and the dependency-free implementations of
Section~\ref{sec:embedded_no_libs}.

\newpage



\begin{thebibliography}{99}

\bibitem{Breger2008}
L.~Breger and J.~P.~How.
\newblock Safe trajectories for autonomous rendezvous of spacecraft.
\newblock \emph{Journal of Guidance, Control, and Dynamics},
  pages 1478--1489, 2008.

\bibitem{Cavenago2019}
F.~Cavenago, P.~Di~Lizia, M.~Massari and A.~Wittig.
\newblock Sequential convex programming for multimode spacecraft trajectory optimisation.
\newblock In \emph{AIAA/AAS Astrodynamics Specialist Conference}, 2019.

\bibitem{Clohessy1960}
W.~H.~Clohessy and R.~S.~Wiltshire.
\newblock Terminal guidance system for satellite rendezvous.
\newblock \emph{Journal of the Aerospace Sciences}, pages 653--658, 1960.

\bibitem{DiCairano2012}
S.~Di~Cairano, H.~Park and I.~Kolmanovsky.
\newblock Model predictive control approach for guidance of spacecraft rendezvous and proximity maneuvering.
\newblock \emph{International Journal of Robust and Nonlinear Control},
  pages 1398--1427, 2012.

\bibitem{Bashnick2023}
C.~Bashnick and S.~Ulrich.
\newblock Fast model predictive control for spacecraft rendezvous and docking with obstacle avoidance.
\newblock \emph{Journal of Guidance, Control, and Dynamics}, 2023.

\bibitem{Ferranti2015}
L.~Ferranti and T.~Keviczky.
\newblock A parallel dual fast gradient method for MPC applications.
\newblock \href{https://arxiv.org/abs/1503.06330}{\textcolor{blue}{arXiv:1503.06330}}, 2015.



\bibitem{Giselsson2013a}
P.~Giselsson.
\newblock Improving fast dual ascent for MPC -- part I: the distributed case.
\newblock \href{https://arxiv.org/abs/1312.3012}{\textcolor{blue}{arXiv:1312.3012}}, 2013.

\bibitem{Giselsson2013b}
P.~Giselsson.
\newblock Improving fast dual ascent for MPC -- part II: the embedded case.
\newblock \href{https://arxiv.org/abs/1312.3013}{\textcolor{blue}{arXiv:1312.3013}}, 2013.

\bibitem{Hartley2015}
E.~N.~Hartley.
\newblock A tutorial on model predictive control for spacecraft rendezvous.
\newblock In \emph{European Control Conference}, 2015.

\bibitem{Herceg2015}
M.~Herceg, M.~Kvasnica, C.~N.~Jones and M.~Morari.
\newblock Tube-based robust model predictive control.
\newblock Tutorial notes, 2015.

\bibitem{Lorenzetti2019}
J.~Lorenzetti and M.~Pavone.
\newblock A simple and efficient tube-based robust output feedback model predictive control scheme.
\newblock \href{https://arxiv.org/abs/1911.07360}{\textcolor{blue}{arXiv:1911.07360}}, 2019.

\bibitem{Mayne2000}
D.~Q.~Mayne, J.~B.~Rawlings, C.~V.~Rao and P.~O.~M.~Scokaert.
\newblock Constrained model predictive control: stability and optimality.
\newblock \emph{Automatica}, pages 789--814, 2000.

\bibitem{Mayne2005}
D.~Q.~Mayne, M.~M.~Seron and S.~V.~Rako\-vi\'{c}.
\newblock Robust model predictive control of constrained linear systems with bounded disturbances.
\newblock \emph{Automatica}, pages 219--224, 2005.

\bibitem{Oestreich2023}
C.~E.~Oestreich, R.~Linares and R.~Gondhalekar.
\newblock Tube-based model predictive control with uncertainty identification for autonomous spacecraft maneuvers.
\newblock \emph{Journal of Guidance, Control, and Dynamics}, pages 1--15, 2023.

\bibitem{Patrinos2016}
N.~Patrinos and A.~Bemporad.
\newblock Proximal gradient methods for MPC.
\newblock Lecture notes, IMT Lucca, 2016.

\bibitem{Rakovic2008}
S.~V.~Rako\-vi\'{c}.
\newblock Tube based model predictive control.
\newblock Tutorial slides, 2008.

\bibitem{Rawlings2017}
J.~B.~Rawlings, D.~Q.~Mayne and M.~Diehl.
\newblock \emph{Model Predictive Control: Theory, Computation, and Design}.
\newblock Nob Hill Publishing, 2nd edition, 2017.

\bibitem{Richter2012}
S.~Richter, C.~N.~Jones and M.~Morari.
\newblock Computational complexity certification for real-time MPC with input constraints based on the fast gradient method.
\newblock \emph{IEEE Transactions on Automatic Control}, pages 1391--1403, 2012.

\bibitem{Specht2023}
C.~Specht, A.~Bishnoi and R.~Lampariello.
\newblock Autonomous spacecraft rendezvous using tube-based model predictive control: design and application.
\newblock \emph{Journal of Guidance, Control, and Dynamics},
  pages 1243--1261, 2023.

\bibitem{Szmuk2025}
M.~Szmuk, B.~A\c{c}{\i}kme\c{s}e and A.~W.~Reynolds.
\newblock Successive convexification for passively-safe spacecraft rendezvous on near rectilinear halo orbit.
\newblock \href{https://arxiv.org/abs/2505.17251}{\textcolor{blue}{arXiv:2505.17251}}, 2025.


\bibitem{Wu2021}
L.~Wu and A.~Bemporad.
\newblock A simple and fast coordinate-descent augmented-Lagrangian solver for model predictive control.
\newblock \href{https://arxiv.org/abs/2109.10205}{\textcolor{blue}{arXiv:2109.10205}}, 2021.

\bibitem{Zamblera2023}
D.~Zamblera.
\newblock Tube model predictive control for robust close-range relative motion.
\newblock Master's thesis, Politecnico di Milano, 2023.

\bibitem{Boyd2011}
S.~Boyd, N.~Parikh, E.~Chu, B.~Peleato and J.~Eckstein.
\newblock Distributed optimization and statistical learning via the alternating direction method of multipliers.
\newblock \emph{Foundations and Trends in Machine Learning}, pages 1--122, 2011.

\bibitem{Richter2009}
S.~Richter, C.~N.~Jones and M.~Morari.
\newblock Real-time input-constrained MPC using fast gradient methods.
\newblock In \emph{Proceedings of the 48th IEEE Conference on Decision and
  Control}, pages 7387--7393, 2009.

\bibitem{Frison2016}
G.~Frison.
\newblock \emph{Algorithms and Methods for Fast Model Predictive Control}.
\newblock PhD thesis, Technical University of Denmark, 2016.

\bibitem{Rao1998}
C.~V.~Rao, S.~J.~Wright and J.~B.~Rawlings.
\newblock Application of interior-point methods to model predictive control.
\newblock \emph{Journal of Optimization Theory and Applications},
  pages 723--757, 1998.

\bibitem{Fehse2003}
W.~Fehse.
\newblock Automated Rendezvous and Docking of Spacecraft.
\newblock Cambridge Aerospace Series. Cambridge University Press, 2003.

\bibitem{Gaias2015}
G.~Gaias, J.-S.~Ardaens and O.~Montenbruck.
\newblock Model of J2 perturbed satellite relative motion with time-varying differential drag.
\newblock \emph{Celestial Mechanics and Dynamical Astronomy},
  pages 411--433, 2015.

\bibitem{Lovell2004}
T.~A.~Lovell and S.~C.~Tragesser.
\newblock Guidance for relative motion of low Earth orbit spacecraft based on relative orbit elements.
\newblock In \emph{AIAA/AAS Astrodynamics Specialist Conference}, 2004.

\bibitem{Nesterov1983}
Y.~Nesterov.
\newblock A method of solving a convex programming problem with convergence rate \(O(1/k^2)\).
\newblock \emph{Soviet Mathematics Doklady}, pages 372--376, 1983.

\bibitem{Stellato2020}
B.~Stellato, G.~Banjac, P.~Goulart, A.~Bemporad and S.~Boyd.
\newblock OSQP: an operator splitting solver for quadratic programs.
\newblock \emph{Mathematical Programming Computation}, pages 637--672, 2020.

\bibitem{Jerez2014}
J.~L.~Jerez, P.~J.~Goulart, S.~Richter, G.~A.~Constantinides,
  E.~C.~Kerrigan and M.~Morari.
\newblock Embedded online optimization for model predictive control at megahertz rates.
\newblock \emph{IEEE Transactions on Automatic Control}, pages 3238--3251,
  2014.


\bibitem{MayneKerrigan2011}
D.~Q.~Mayne, E.~C.~Kerrigan, E.~J.~van~Wyk and P.~Falugi.
\newblock Tube-based robust nonlinear model predictive control.
\newblock \emph{International Journal of Robust and Nonlinear Control},
  pages 1341--1353, 2011.

\bibitem{MayneKerrigan2007}
D.~Q.~Mayne and E.~C.~Kerrigan.
\newblock Tube-based nonlinear model predictive control.
\newblock \emph{IFAC Proceedings Volumes}, pages 36--41, 2007.
\newblock 7th IFAC Symposium on Nonlinear Control Systems, Pretoria,
  South Africa.

\bibitem{Zhang2022Koopman}
X.~Zhang, W.~Pan, R.~Scattolini, S.~Yu and X.~Xu.
\newblock Robust tube-based model predictive control with Koopman operators.
\newblock \emph{Automatica}, article 110114, 2022.

\bibitem{Zhou2024Delay}
L.~Zhou, S.~Ma, L.~Cen, J.~Ma and T.~Peng.
\newblock Tube-based model predictive control for linear systems with bounded disturbances and input delay.
\newblock \emph{ISA Transactions}, pages 56--66, 2024.

\bibitem{DoffSotta2026}
M.~Doff-Sotta, Z.~A-Rahman and M.~Cannon.
\newblock Computationally tractable robust nonlinear model predictive
  control using DC programming.
\newblock \href{https://arxiv.org/abs/2602.01164}{\textcolor{blue}{arXiv:2602.01164}}, 2026.

\bibitem{Tranos2022}
D.~Tranos, A.~Russo and A.~Proutiere.
\newblock Self-tuning tube-based model predictive control.
\newblock \href{https://arxiv.org/abs/2210.00502}{\textcolor{blue}{arXiv:2210.00502}}, 2022.

\bibitem{SzmukAcikmese2018}
M.~Szmuk and B.~A\c{c}{\i}kme\c{s}e.
\newblock Successive convexification for 6-DoF Mars rocket powered landing with free-final-time.
\newblock In \emph{AIAA Guidance, Navigation, and Control Conference},
  2018.

\bibitem{Massera2020}
C.~M.~Massera, M.~H.~Terra and D.~F.~Wolf.
\newblock Tube-based guaranteed cost robust model predictive control for
  linear systems subject to parametric uncertainties.
\newblock \href{https://arxiv.org/abs/2012.05349}{\textcolor{blue}{arXiv:2012.05349}}, 2020.

\bibitem{Parsi2022}
A.~Parsi, A.~Iannelli and R.~S.~Smith.
\newblock Scalable tube model predictive control of uncertain linear systems using ellipsoidal sets.
\newblock \emph{International Journal of Robust and Nonlinear Control},
  2022.

\bibitem{Kempf2020}
I.~Kempf, P.~Goulart and S.~Duncan.
\newblock Fast gradient method for model predictive control with input rate and amplitude constraints.
\newblock \emph{IFAC-PapersOnLine}, pages 6542--6547, 2020.
\newblock 21st IFAC World Congress.

\bibitem{Herceg2013}
M.~Herceg, M.~Kvasnica, C.~N.~Jones and M.~Morari.
\newblock Multi-parametric toolbox 3.0.
\newblock In \emph{European Control Conference (ECC)}, pages 502--510,
  2013.

\bibitem{Langson2004}
W.~Langson, I.~Chryssochoos, S.~V.~Rakovi\'c and D.~Q.~Mayne.
\newblock Robust model predictive control using tubes.
\newblock \emph{Automatica}, pages 125--133, 2004.

\bibitem{Rakovic2005RPI}
S.~V.~Rakovi\'c, E.~C.~Kerrigan, K.~I.~Kouramas and D.~Q.~Mayne.
\newblock Invariant approximations of the minimal robust positively invariant set.
\newblock \emph{IEEE Transactions on Automatic Control}, pages 406--410, 2005.

\bibitem{Hartley2012}
E.~N.~Hartley, J.~M.~Maciejowski, H.~Cervantes and L.~Kondiakos.
\newblock Model predictive control system design and implementation for spacecraft rendezvous.
\newblock \emph{Control Engineering Practice}, pages 695--713, 2012.

\bibitem{Hartley2013}
E.~N.~Hartley, M.~Gallieri and J.~M.~Maciejowski.
\newblock Terminal spacecraft rendezvous and capture with LASSO model predictive control.
\newblock \emph{International Journal of Control}, pages 2104--2113, 2013.



\bibitem{Golub1996}
G.~H.~Golub and C.~F.~Van~Loan.
\newblock Matrix Computations.
\newblock Johns Hopkins University Press, 3rd edition, 1996.

\bibitem{Benoit1924}
Commandant Beno\^it. Note sur une m\'ethode de r\'esolution des \'equations
normales provenant de l'application de la m\'ethode des moindres carr\'es
aux probl\`emes de triangulation. \textit{Bulletin G\'eod\'esique}, 1924.

\end{thebibliography}
\end{document}